\documentclass[11pt]{amsart} \textwidth=14.5cm \oddsidemargin=1cm
\usepackage[utf8]{inputenc}

\title{Quantum supersymmetric pairs and $\imath$Schur duality of type AIII}
\author{Yaolong Shen}
\usepackage{setspace}
\usepackage{amsthm}
\usepackage[rgb]{xcolor}
\usepackage{times}
\usepackage[linktocpage=true]{hyperref}
\usepackage{mathrsfs}  
\usepackage{xypic}
\usepackage{tikz}
\usepackage{verbatim}
\usepackage{amscd}
\usepackage{amsmath}
\usepackage{amsfonts}
\usepackage{amsxtra}
\usepackage{amssymb}
\usepackage{eucal}
\usepackage{latexsym,todonotes}
\usepackage{stmaryrd}
\usepackage{graphicx}%
\usepackage{dynkin-diagrams}
\usepackage{tikz}
\usetikzlibrary{intersections}
\usetikzlibrary{shapes}
\newtheorem{theorem}{Theorem}[section]

\newtheorem{corollary}[theorem]{Corollary}

\newtheorem{example}[theorem]{Example}

\newtheorem{lemma}[theorem]{Lemma}

\newtheorem{proposition}[theorem]{Proposition}

\numberwithin{equation}{section}

\theoremstyle{remark}
\newtheorem{remark}[theorem]{Remark}

\newcommand*\emptycirc[1][1ex]{\tikz\draw (0,0) circle (#1);} 
\newcommand*\halfcirc[1][1ex]{%
	\begin{tikzpicture}
	\draw[fill] (0,0)-- (90:#1) arc (90:270:#1) -- cycle ;
	\draw (0,0) circle (#1);
	\end{tikzpicture}}
\newcommand*\fullcirc[1][1ex]{\tikz\fill (0,0) circle (#1);}

\newcommand{\ad}{\text{ad}}

\newcommand{\N}{\mathbb{N}}
\newcommand{\Z}{\mathbb{Z}}
\newcommand{\R}{\mathbb{R}}

\newcommand{\Q}{\mathbb{Q}}
\newcommand{\C}{\mathbb{C}}

\newcommand{\Hy}{\mathscr{H}}

\newcommand{\HB}{\mathscr H_{B_d}}

\newcommand{\End}{\text{End}\ }

\newcommand{\la}{\langle}
\newcommand{\ra}{\rangle}

\newcommand{\io}{\imath}
\newcommand{\bu}{\bullet}
\newcommand{\U}{\mathbf U}
\newcommand{\Ui}{{\mathbf U}^\imath}
\newcommand{\Uinew}{{\mathbf U}^{\imath}}

\newcommand{\va}{\varsigma}

\newcommand{\iba}{\psi_{\io}}
\newcommand{\gl}{\mathfrak g \mathfrak l}

\newcommand{\W}{\mathbb W}
\newcommand{\pt}{n}
\newcommand{\new}{\varrho}
\newcommand{\nb}{m}
\newcommand{\up}{\Upsilon}
\newcommand{\I}{\mathbb I}
\newcommand{\Ibw}{\I_{r|m|r}}

\newcommand{\Iwl}{\mathbb I_{\circ}^{-}}
\newcommand{\Iwr}{\mathbb I_{\circ}^{+}}
\newcommand{\Ib}{\mathbb I_{\bu}}
\newcommand{\Veven}{V_{\overline 0}}
\newcommand{\Vodd}{V_{\overline 1}}

\newcommand{\geven}{\mathfrak g_{\overline 0}}
\newcommand{\godd}{\mathfrak g_{\overline 1}}
\newcommand{\Ieven}{I_{\overline 0}}
\newcommand{\Iodd}{I_{\overline 1}}
\newcommand{\Phieven}{\Phi_{\overline 0}}
\newcommand{\Phiodd}{\Phi_{\overline 1}}
\newcommand{\Uio}{\U^{\imath0}}
\newcommand{\Ub}{\U_{\bu}}
\newcommand{\wl}{P_\io}
\newcommand{\cwl}{P^\vee_\io}
\newcommand{\qbinom}[2]{\begin{bmatrix} #1\\#2 \end{bmatrix} }

\newcommand{\lskew}{{}_ir}
\newcommand{\rskew}{r_i}
\newcommand{\bskew}{\overline {\rskew}}

\begin{document}
\maketitle
\begin{abstract}
We construct quantum supersymmetric pairs $({\bold U},{\bold U}^\imath)$ of type AIII and elucidate their fundamental properties. An $\imath$Schur duality between the $\imath$quantum supergroup ${\bold U}^\imath$ and the Hecke algebra of type B acting on a tensor space is established, providing a super generalization of the $\imath$Schur duality of type AIII. Additionally, we construct a (quasi) $K$-matrix for arbitrary parameters, which facilitates the realization of the Hecke algebra action on the tensor space.
\end{abstract}

\tableofcontents

\section{Introduction}

\subsection{Background}

Let $\mathfrak{g}$ be a Lie algebra of finite type, and let $\theta$ be an involution on $\mathfrak{g}$. The theory of quantum symmetric pairs $(\mathbf{U}_q(\mathfrak{g}), \mathbf{U}^{\imath})$, which provides a quantization of the symmetric pair $(\mathfrak{g}, \mathfrak{g}^\theta)$, was systematically developed by Letzter \cite{Let99,Let02}. In this context, $\mathbf{U}_q(\mathfrak{g})$ represents the Drinfeld-Jimbo quantum group associated with $\mathfrak{g}$, while $\mathbf{U}^\imath$ denotes the $\imath$quantum group. Kolb \cite{Ko14} further expanded and generalized this theory to cover the Kac-Moody case. Some early examples of quantum symmetric pairs were constructed by Noumi and his collaborators, cf. \cite{N96,NS95}.


Jimbo \cite{Jim86} established the Schur-Jimbo duality, which relates the type A quantum group and the Hecke algebra of type A. This duality quantizes the Schur duality between the general linear group and the symmetric group. Over the years, the Schur-Jimbo duality has been extended naturally in conjunction with the development of $\imath$quantum groups. In \cite{BW18b}, the authors demonstrate that the Hecke algebra of type B and the $\imath$quantum group of type AIII satisfy a double centralizer property (see also \cite{Bao17}). Furthermore, it was shown that the Kazhdan-Lusztig basis of type B coincides with the $\imath$-canonical bases \cite{BW18b} arising from tensor product modules of $\imath$quantum groups. This result was later generalized to a multi-parameter setting \cite{BWW18}. More recently, a unified generalization of both type A and type B Schur dualities has been constructed in \cite{SW23}.

A fundamental property of a quantum symmetric pair is the existence of the quasi $K$-matrix. The quasi $K$-matrix for a quantum symmetric pair was first introduced in \cite[\S 2.3]{BW18a} as the intertwiner between the embedding of the $\imath$quantum group into the underlying quantum group and its bar-conjugated embedding with some conditions imposed on the parameters; also cf. \cite{BK19}. In \cite{AV22}, the quasi $K$-matrix was reformulated by the authors without invoking the bar involution, thereby allowing for more general parameters. More recently, Wang and Zhang presented a conceptual reformulation of the quasi $K$-matrix under general parameters as an intertwiner of anti-involutions in \cite{WZ22}.

In the Schur-Jimbo duality, it is well known that the action of the generators $H_i\ (1\leqslant i\leqslant d-1)$ of the type A Hecke algebra can be realized via the $R$-matrix, cf. \cite{Jim86}. In \cite[Theorem 2.18]{BW18a}, the authors constructed a $K$-matrix $\mathcal T$ which is an $\Ui$-module isomorphism for any finite-dimensional $\U$-module (see also \cite{BK19} for generalizations). Furthermore, Bao and Wang showed in \cite[Lemma 5.3]{BW18a} that the $H_0$-action on the tensor space can be realized via the $K$-matrix $\mathcal T$. This result was further generalized in \cite[\S 5.4]{SW23}.
 
 From now on we fix $\mathfrak g=\gl(\mathfrak m|\mathfrak n)$ to be a Lie superalgebra of type A; a well-known fact is that the Dynkin diagrams of $\mathfrak g$ are not unique (cf. \cite{CW12}). Sergeev \cite{S85} has extended the Schur duality in the setting of $\gl(\mathfrak m|\mathfrak n).$ The {\em quantum supergroup} $\U$, as a Drinfeld-Jimbo quantization of $\mathfrak g$, has been defined in \cite{Ya94} associated to any Dynkin diagram of $\mathfrak g$. Moreover, it was shown in \cite{Mi06} that the type A quantum supergroup associated to the standard Dynkin diagram and the Hecke algebra of type A satisfy a double centralizer property.

Lusztig's braid group action \cite{Lus93} is a fundamental construction for usual quantum groups, where the braid group operators quantize the reflections on the weight data. However, for Lie superalgebras, the fundamental systems of the root system associated with $\mathfrak g$ are not conjugated under the Weyl group actions due to the existence of odd roots (cf. \cite{CW12}). The reflections associated with the odd (resp. even) roots are known as odd (resp. even) reflections.

In the context of Lie superalgebras, Yamane \cite{Ya99} quantized the odd reflections into algebra isomorphisms of $\U$ associated with different presentations, providing a super analogue of Lusztig's braid group operators. Additionally, in \cite{C16}, the author reformulated Yamane's results to introduce braid group operators in the setting of $\gl(\mathfrak m|1)$.
 
 On the other hand, Kolb and Yakimov \cite{KY20} have extended Letzter’s theory of quantum symmetric
 pairs and constructed the $K$-matrix for Drinfeld doubles of pre-Nichols
algebras of diagonal type, which contains quantum supersymmetric pair of quasi-split type AIII as a special case. Moreover, Chung \cite{Ch19} has studied quantum symmetric pairs $(\U_{\pi},\Ui_{\pi})$ for quantum covering algebras $\U_{\pi}$ which is introduced in \cite{CHW13} and specializes to the Lusztig quantum group when $\pi=1$ and quantum supergroups of anisotropic type when $\pi=-1$; see also \cite{Ch21}.
 

\subsection{Goal} \label{subsec:goal}
Let $(\mathfrak g,\mathfrak g^\theta)$ be a {\em supersymmetric pair} of type AIII; see  \eqref{eq:AIIIdiagram}. Our first goal is to construct a suitable quantization $(\U,\Uinew)$ of $(\mathfrak g,\mathfrak g^\theta)$, following the approach in \cite{Let99,Let02}. This construction is known as a {\em quantum supersymmetric pair}, where $\Uinew$ is referred to as an {\em $\imath$quantum supergroup}.

A type AIII Satake diagram \eqref{eq:AIIIdiagram} consists of $(I=I_\bu\cup I_\circ,\tau)$ where $\tau$ is the diagram involution. For example, $(\gl(\mathfrak m_1+ \mathfrak m_2|\mathfrak n_1+\mathfrak n_2),\gl(\mathfrak m_1|\mathfrak n_1)\oplus \gl(\mathfrak m_2|\mathfrak n_2))$ is a supersymmetric pair of type AIII. Note that the Satake diagrams we consider in this paper allow odd simple roots to appear in both $I_\circ$ and $I_\bu$.

In order to achieve the first goal, we find it necessary to examine $\U$ associated with various Dynkin diagrams, as the definition of $\Ui$ incorporates braid group operators linked to all simple roots in $I_\bu$. The construction will be greatly simplified if we assume that $I_\bu$ consists of even simple roots only. It is worth noting that even with this assumption, the results presented in this work encompass many new cases.

Our subsequent objective is to establish a Schur duality between $\Ui$ and the Hecke algebra of type B, which acts on the tensor module of the natural representation of $\U$. This Schur duality generalizes the $\imath$Schur dualities found in \cite{BW18a,BWW18,SW23}.

Our third goal is to construct the quasi $K$-matrix $\up$ associated to $(\U,\Ui)$ generalizing the conceptual reformulation in \cite{WZ22}; also cf. \cite{BW18a,BK19}. The quasi $K$-matrix plays a fundamental role for a quantum supersymmetric pair as the quasi $R$-matrix for a quantum supergroup; cf. \cite{KT91}. Under some assumptions on the parameters, we use the quasi $K$-matrix to construct a bar involution on $\Ui$ following \cite{Ko22}.

Our last goal is to construct a $K$-matrix $\mathcal T$ following \cite{BW18b}. Moreover, after imposing special parameters so that the $\imath$Schur duality holds, we can show that $\mathcal T$ realizes the $H_0$-action on the tensor module of the fundamental representations of $\U$. This gives us a conceptual explanation on the commutativity between the $H_0$-actions and the actions of $\Ui$.

\subsection{Main results}

We first generalize Clark's result on $\gl(\mathfrak m|1)$ to $\gl(\mathfrak m|\mathfrak n)$ setting (see Theorem \ref{thm:braid}) by checking the braid group operators are indeed algebra isomorphisms between the same algebra with different presentations. The braid group operators are usually denoted by $T_j$ in this paper and are known to satisfy the braid group relations, cf. \cite{H10}.

To define a quantum supersymmetric pair $(\U,\Ui)$ of type AIII, we employ the braid group operators. We start with a diagram $X$ that satisfies the conditions \eqref{eq:AIIIdiagram} and \eqref{eq:assumption}, where the index set is denoted as $I=I_\circ \cup I_\bu$. Taking into account the presence of odd reflections, we can mimic the non-super case to define a longest element $w_\bu$ associated with the Weyl groupoid of the Levi subalgebra corresponding to $I_\bu$ (cf. \cite{HY08}). Applying $w_\bu$ to $X$ results in another diagram $Y$ that satisfies \eqref{eq:assumption} and is indexed by $I$ as well. The algebra $\U(Y)$ (and $\U(X)$) is generated by $E_j,F_j$ (and $E^{\texttt {X}}_j, F^{\texttt {X}}_j$) along with the Cartan part. The $\imath$quantum group $\Ui(Y)$ in the pair $(\U(Y),\Ui(Y))$ is generated by $E_j, F_j \ (j\in I_\bu)$,
\[B_j= 
 F_j+\va_jT_{w_\bu}(E^{\texttt{X}}_j)  K_j^{-1},\ \ \text{for } j\in I_\circ\] 
 together with certain Cartan elements. 
 
Note that for simplicity, a reader can consider the case when $I_\bu$ consists only of even simple roots. This simplification helps avoid complexity, and even in this case, we still obtain significant new results. Notably, when $I_\bu$ exclusively consists of even simple roots, the diagrams $Y$ and $X$ coincide, and we only need to employ Lusztig's braid group actions in the definition of $B_j$.

Additionally, we establish the coideal subalgebra property of $\Ui$ in $\U$ (see Proposition \ref{prop:coideal}), which generalizes the non-super constructions in \cite{Let99,Let02,Ko14}. We note that the original methods used in \cite{Let99,Let02,Ko14} do not directly apply to the super case. As a result, we provide a new proof specifically tailored for the super type AIII case.


Subsequently, we establish new quantum $\imath$Serre relations (Proposition \ref{prop:BSerre1}) by employing the projection technique introduced in \cite{Ko14}. This allows us to obtain a natural filtration on $\Ui$ (Proposition \ref{prop:filtration}), where the associated graded algebra $gr \Ui$ is essentially isomorphic to a parabolic subalgebra of $\U$ modulo the Cartan part. Moreover, we establish the quantum Iwasawa decomposition of $\U$ with respect to $\Ui$ in Theorem~\ref{thm:Iwasawa}.

Having established the quantum supersymmetric pair $(\U,\Uinew)$, we proceed to establish a multi-parameter $\imath$Schur duality of type AIII between $\Uinew$ and $\HB$, the Hecke algebra of type B (cf. \cite{BW18a,BWW18,SW23}).

Let $\W$ be the natural representation of $\U$. We show that $\W^{\otimes d}$ possesses a right $\HB$-module structure (Proposition \ref{prop:Heckeaction}) and a left $\Ui$-module structure via the comultiplication in the same time. Under the assumption on the parameters \eqref{eq:para}, we show that the actions of $\Ui$ and $\HB$ commute with each other and form a double centralizer property; see Theorem \ref{thm:UiHB}.

In \cite{CL22}, a duality between type AIII $\imath$Schur superalgebras and the Hecke algebra of type B was investigated. This duality is applicable to a very special Satake diagram where $I_\bu=\varnothing$. Notably, our $\Ui$-actions factor through their $\imath$Schur superalgebra in cases where the corresponding diagrams are identical.

For the construction of the quasi $K$-matrix $\up$ and the $K$-matrix $\mathcal T$ we impose one more condition \eqref{eq:extraassumption}; i.e. $I_\bu$ consists of even simple roots only. Following intertwining relations from \cite{WZ22} and strategies from \cite{BW18a,BK19}, we construct the quasi $K$-matrix $\up=\sum\up_\mu$ in a completion of $\U$ with $\up_0=1$ and $\up_\mu\in \U^{+}_\mu$.

With the quasi $K$-matrix $\up$ being established, we impose one more constraint on the parameters and construct a unique involutive bar involution $\iba$ on $\Ui$  (Corollary \ref{cor:iba}), which is a super analogue of the bar involution established in \cite{BW18b,Ko22}. The bar involution $\psi$ on the quantum supergroup $\U$ and $\iba$ on $\Ui$ is intertwined by $\up$ such that $\iba(x)\up=\up \psi(x)$, for all $x\in \Ui$.

Finally, following the construction presented in \cite{BW18b}, we formulate the $K$-matrix $\mathcal T$. The construction by Bao and Wang for the $K$-matrix involves braid group operators associated with $I_\bu$ and is applicable to quantum symmetric pairs of Kac-Moody types as well as our situation. We demonstrate that $\mathcal T$ induces an $\Ui$-isomorphism on $\W$ and compute its action on $\W$, which coincides with the $H_0$-action. Consequently, the $H_0$-action on $\W^{\otimes d}$ is realized by $\mathcal T\otimes 1^{\otimes d-1}$.

\subsection{Organization} This paper is organized as follows. In Section \ref{sec:QSG}, we present a fundamental construction of the quantum supergroup $\U$ of type A for any Dynkin diagram. In Section \ref{sec:BGO}, we examine the odd reflections of the underlying Lie superalgebra and extend Clark's braid group operators to the general case. Subsequently, in Section \ref{sec:QSP}, we define the quantum supersymmetric pair of type AIII and investigate new $\imath$Serre relations and the quantum Iwasawa decomposition. In Section \ref{sec:duality}, we establish the bimodule structure on the tensor product of the natural representation of $\U$ and prove the double centralizer property. Furthermore, in Section \ref{sec:Kmatrix}, we construct the quasi $K$-matrix of $(\U,\Ui)$ and introduce the bar involution on $\U$. Lastly, in Section \ref{sec:UK}, we establish the $K$-matrix $\mathcal T$ and utilize it to realize the $H_0$-action on the natural representation of $\U$.
\vspace{2mm}

{\bf Acknowledgement.} The author thanks his advisor Weiqiang Wang for many insightful advice and helpful discussions. The author also thanks Weinan Zhang for his useful suggestions. YS is supported by Graduate Research Assistantship from Wang's NSF grant (DMS-2001351) and a semester fellowship from University of Virginia.

\section{Quantum supergroup of type A}
\label{sec:QSG}
In this section we  provide an overview of the type A Lie superalgebra and the quantum supergroup associated to any Dynkin diagram of type A.

\subsection{The general linear Lie superalgebra}
\label{sec:2.1}
We adopt basic notations from \cite{CW12}. Let $V=\Veven \oplus \Vodd$ be a vector superspace such that $End(V)$ is an associative superalgebra. Then $End(V)$, equipped with the supercommutator, forms a Lie superalgebra, called the general linear Lie superalgebra and is denoted by $\gl(V)$ or $\gl(\mathfrak m|\mathfrak n)$ where $dim \Veven=\mathfrak m,\ dim\Vodd=\mathfrak n$.

Choose bases for $\Veven$ and $\Vodd$ such that they combine to a homogeneous basis of $V$. We will make it a convention to parameterize such a basis by the set
\begin{equation}
\label{eq:Imn}
    I(\mathfrak m|\mathfrak n)=\{\overline 1,\ldots,\overline {\mathfrak m},\underline 1,\ldots,\underline {\mathfrak n}\}
\end{equation}
with total order
$$\overline 1<\cdots<\overline {\mathfrak m}<0<\underline 1<\cdots \underline {\mathfrak n}.$$
Here $0$ is inserted for convention. With such an ordered basis, $\gl(\mathfrak m|\mathfrak n)$ can be realized as $(\mathfrak m+\mathfrak n)\times (\mathfrak m+\mathfrak n)$ complex matrices of the block form
\begin{equation}
    \label{eq:matrixform}
g=\begin{pmatrix}
a & b\\
c & d
\end{pmatrix}
\end{equation}
where $a,b,c$ and $d$ are respectively $\mathfrak m\times \mathfrak m,\mathfrak m\times \mathfrak n,\mathfrak n\times \mathfrak m$ and $\mathfrak n\times \mathfrak n$ matrices.

The even subalgebra $\geven$ consists of matrices of the form \eqref{eq:matrixform} with $b=c=0$, while the odd subspace $\godd$ consists of those with $a=d=0$. For each element $g$, we define the supertrace to be $$str(g)=tr(a)-tr(d).$$

The supertrace $str$ on the general linear Lie superalgebra gives rise to a non-degenerate supersymmetric bilinear form
$$(\cdot,\cdot):\gl(\mathfrak m|\mathfrak n)\times \gl(\mathfrak m|\mathfrak n)\to \C, \quad (x,y)=str(xy).$$

\subsection{Root system of $\gl(\mathfrak m|\mathfrak n)$}
Let $\mathfrak g=\gl(\mathfrak m|\mathfrak n)$ and $\mathfrak h$ be the Cartan subalgebra of diagonal matrices.

Restricting the supertrace to the Cartan
subalgebra $\mathfrak h$, we obtain a non-degenerate symmetric bilinear form on it. Denote by $\{\epsilon_{a}\}_{a\in I(\mathfrak m|\mathfrak n)}$ the basis of $\mathfrak h^*$ dual to the set of standard matrices $\{E_{a,a}\}_{a\in I(\mathfrak m|\mathfrak n)}$. Its root system $\Phi=\Phieven\oplus \Phiodd$ is given by
\begin{equation}
\begin{aligned}
    \Phieven=&\{\epsilon_a-\epsilon_b\mid a\neq b\in I(\mathfrak m|\mathfrak n), a,b>0 \text{ or } a,b<0\}, \\
     \Phiodd=&\{\pm(\epsilon_a-\epsilon_b)\mid a<0<b\}.
\end{aligned}
\end{equation}

A fundamental system of $\gl(\mathfrak m|\mathfrak n)$ consists of $\mathfrak m+\mathfrak n-1$ roots $$\epsilon_{i_1}-\epsilon_{i_2},\ldots,\epsilon_{i_{\mathfrak m+\mathfrak n-1}}-\epsilon_{i_{\mathfrak m+\mathfrak n}},$$
where $\{i_1,\ldots,i_{\mathfrak m+\mathfrak n}\}=I(\mathfrak m|\mathfrak n)$. We denote even simple roots by $\bullet$ and odd simple roots by $
\bigotimes$. Then the corresponding Dynkin diagram is of the form
\begin{equation}
\label{eq:dynkin}
\begin{tikzpicture}[scale=1, semithick]
\node (1) [circle,draw,label=below:{$\epsilon_{i_{1}}-\epsilon_{i_{ 2}}$},scale=0.6] at (0,0){.};
\node (2) [circle,draw,label=below:{$\epsilon_{i_{ 2}}-\epsilon_{i_{ 3}}$},scale=0.6] at (1.5,0){.};
\node (3)  at (3,0) {$\cdots$} ;
\node (4) [circle,draw,scale=0.6] at (4.5,0){.};
\node (5) [circle,draw,label=below:{$\epsilon_{i_{\mathfrak m+\mathfrak n-1}}-\epsilon_{i_{\mathfrak m+\mathfrak n}}$},scale=0.6] at (6,0){.};

\path (1) edge (2)
          (2) edge (3)
          (3) edge (4)
          (4) edge (5);
\end{tikzpicture}
\end{equation}
where $\bigodot$ is either $\bullet$ or $\bigotimes$.

\begin{example}
\label{ex:standard}
The standard Dynkin diagram is given by
\begin{equation*}
\begin{tikzpicture}[scale=1, semithick]
\node (1) [fill,circle,draw,label=below:{$\epsilon_{\overline 1}-\epsilon_{\overline{2}}$},scale=0.6] at (0,0){};
\node (2) [fill,circle,draw,label=below:{$\epsilon_{{\overline2}}-\epsilon_{\overline 3}$},scale=0.6] at (1.5,0){};
\node (3)  at (3,0) {$\cdots$} ;
\node (4) [label=below:{$\epsilon_{\overline {\mathfrak m}}-\epsilon_{\underline 1}$},scale=0.8] at (4.5,0){$\bigotimes$};
\node (5)  at (6,0) {$\cdots$} ;
\node (6) [fill,circle,draw,label=below:{},scale=0.6] at (7.5,0){};
\node (7) [fill,circle,draw,label=below:{$\epsilon_{\underline {\mathfrak n-1}}-\epsilon_{\underline {\mathfrak n}}$},scale=0.6] at (9,0){};

\path (1) edge (2)
          (2) edge (3)
          (3) edge (4)
          (4) edge (5)
          (5) edge (6)
          (6) edge (7);
\end{tikzpicture}
\end{equation*}
\end{example}

Given a Dynkin diagram of the form \eqref{eq:dynkin}. Let $$\Pi=\{\alpha_j=\epsilon_{i_j}-\epsilon_{i_{j+1}}\mid j=1,\cdots,\mathfrak m+\mathfrak n-1\}$$ denote the set of simple roots with the index set $I=\{1,\ldots, \mathfrak m+\mathfrak n-1\}$. We see that $I$ is a disjoint union of two subsets $I=\Ieven \cup \Iodd$ where $\Ieven$ (resp. $\Iodd$) consists of all even (resp. odd) simple roots. Let $p$ be the parity function on $I$ such that
\begin{align*}
    p(k)=\begin{cases}
    0 & \text{ if } k\in \Ieven,\\
    1 & \text{ if } k\in \Iodd.
    \end{cases}
\end{align*}

We define the weight lattice $P=\oplus_{b\in I(\mathfrak m|\mathfrak n)}\Z\epsilon_b$ while the symmetric bilinear form on $P$ is given by
\begin{equation}
    (\epsilon_{a},\epsilon_{a'})=
    \begin{cases}
    1 & \text{ if }  a=a'<0, \\
    -1 & \text{ if } a=a'>0, \\
    0 & \text{ else.}
    \end{cases}
\end{equation}

We define the coweight lattice $P^\vee=\oplus_{b\in I(\mathfrak m|\mathfrak n)}\Z\epsilon^{\vee}_{b}$ and we have the pairing $\la \cdot,\cdot \ra:P^\vee\times P\to \Z$ with $\la \epsilon^\vee_a,\epsilon_b \ra=\delta_{a,b}.$ Then $\Pi^\vee=\{h_j\mid j\in I\}$, the set of simple coroots, is given by 
\begin{align}
    h_j=\epsilon_{i_j}^\vee-(-1)^{p(j)}\epsilon_{i_{j+1}}^\vee.
\end{align}

We also define the root lattice $Q:=\Z \Pi$ and the coroot lattice $Q^\vee:=\Z \Pi^\vee$. The generalized Cartan matrix $A=(a_{ij})_{i,j\in I}$ associated with $\mathfrak{g}$ is defined by $a_{ij}=\langle h_j, \alpha_i \rangle$. We observe that $A$ is symmetrizable, meaning that there exist non-zero integers $\ell_j$ satisfying
\begin{align}
\label{eq:lj}
\ell_j\langle h_j, \lambda \rangle = (\alpha_j, \lambda) \quad \text{for any } \lambda \in P.
\end{align}
When $p(j)=0$, we see that $\ell_j=\frac{(\alpha_j, \alpha_j)}{2}$.

\subsection{Quantum supergroup of type A}

Following \cite{Ya94}, we define a quantum supergroup associated to any fixed Dynkin diagram of the form \eqref{eq:dynkin}. 

Denote the quantum integers and quantum binomial coefficients by, for $a\in \Z, k \in \N$,
\[
[a]=\frac{q^a-q^{-a}}{q-q^{-1}},
\qquad
\qbinom{a}{k} =\frac{[a] [a-1] \ldots [a-k+1]}{[k]!}.
\]

It will be convenient for us to introduce the following notation. We will say $i,j\in I$ are {\em connected} if $i=j\pm 1$ and write $i\sim j$. Likewise, we say {\em not connected} if $i\neq j,j\pm 1$ and write $i\nsim j$.

Let $K_i=q^{\ell_ih_i}$, we recall the definition of $\U_q(\gl(\mathfrak m|\mathfrak n))$ to be the unital associative algebra over $\Q(q)$ with generators $q^h\ (h\in P^\vee), E_i, F_i\ (i\in I)$ which satisfy the following defining relations:
\begin{equation}
\label{eq:Urelation}
\begin{aligned}
    &(R1)\  q^h=1,\quad\text{ for } h=0,\\
    &(R2)\  q^{h_1}q^{h_2}=q^{h_1+h_2},\\
    &(R3)\  q^h E_j=q^{\la h,\alpha_j \ra}E_j q^h \quad \text{ for }j\in I,\\
    &(R4)\  q^h F_j=q^{-\la h,\alpha_j \ra}F_j q^h,\quad \text{ for }j\in I,\\
     &(R5)\  [E_j,F_k]=E_j F_k-(-1)^{p(j)p(k)}F_k E_j=\delta_{j,k}\frac{K_j-K_j^{-1}}{q^{\ell_j}-q^{-\ell_j}},\quad \text{ for }j,k\in I,\\
    &(R6)\  E_j^2=F_j^2=0,\quad \text{ for }j\in I_{\overline 1},\\
    &(R7)\  E_j E_k=(-1)^{p(j)p(k)}E_k E_j,\quad F_j F_k=(-1)^{p(j)p(k)}F_k F_j, \text{ for }j\nsim k,\\
    &(R8)\  E_j^2E_k-[2]E_j E_k E_j+E_k E_j^2=0,\quad \text{ for }j\sim k, p(j)=0,\\
    &(R9)\  F_j^2F_k-[2]F_j F_k F_j+F_k F_j^2=0,\quad \text{ for }j\sim k, p(j)=0,\\
    &(R10)\  S_{p(k),p(\ell)}(E_k,E_j,E_\ell)=0,
    \quad \text{ for }k\sim j\sim \ell,k<\ell, p(j)=1,\\
    &(R11)\  S_{p(k),p(\ell)}(F_k,F_j,F_\ell)=0,
    \quad \text{ for }k\sim j\sim \ell,k<\ell, p(j)=1.
\end{aligned}
\end{equation}
where $S_{t_1,t_2}(x_1,x_2,x_3)\in \Q(q)\la x_1,x_2,x_3\ra$ is the polynomial in three non-commuting variables for $t_1,t_2\in\{\overline 0,\overline 1\}$ given by
\begin{equation}
    \label{eq:superSerre}
\begin{aligned}
    &S_{t_1,t_2}(x_1,x_2,x_3)=[2]x_2x_3x_1x_2-[((-1)^{t_1}x_2x_3x_2x_1+(-1)^{t_1+t_1t_2}x_1x_2x_3x_2)\\
    &+((-1)^{t_1t_2+t_2}x_2x_1x_2x_3+(-1)^{t_2}x_3x_2x_1x_2)].
\end{aligned}
\end{equation}


Moreover, let $q_j:=q^{\ell_j}$, we define maps $\sigma,\wp, \overline{\cdot}$ on $\U_q(\gl(\mathfrak m|\mathfrak n))$ satisfying:
\begin{equation}
\label{eq:rhoandop}
\begin{aligned}
&\sigma(E_j)=E_j,\quad\sigma(F_j)=F_j,\quad\sigma(K_j)=(-1)^{p(j)}K_j^{-1},\quad \sigma(xy)=\sigma(y)\sigma(x),\\
&\wp(E_j)=q_j K_j F_j,\quad \wp(F_j)=q_j^{-1}E_jK_j^{-1},\quad \wp(K_j)=K_j,\quad \wp(xy)=\wp(y)\wp(x),\\
&\overline{E_j}=E_j,\quad \overline{F_j}=F_j,\quad
\overline{K_j}=K^{-1}_j,\quad
\overline{q}=q^{-1},\quad
\overline{xy}=\overline{x}\cdot \overline{y}.
\end{aligned}
\end{equation}

For convention we also write
\begin{equation}
\label{eq:KQ}
K_\lambda:=\prod_{i\in I}K_i^{n_i},\quad \text { for any }\lambda=\sum_{i\in I}n_i\alpha_i\in Q.\end{equation}

In general, $\U_q(\gl(\mathfrak m|\mathfrak n))$ is a Hopf superalgebra (cf. \cite[Lemma 2.1]{C16}) but not a Hopf algebra. We define an involutive operator $\new$ of parity $0$ on it by
\begin{equation}
\label{eq:new}
    \new(q^h)=q^h,\quad \new(E_j)=(-1)^{p(j)}E_j\text{ and }
    \new(F_j)=(-1)^{p(j)}F_j,\quad \forall j\in I.
\end{equation}

Let \[\U=\U_q(\gl(\mathfrak m|\mathfrak n))\oplus \U_q(\gl(\mathfrak m|\mathfrak n))\new.\] Then $\U$ is an algebra with the additional multiplication law given by 
\begin{equation}
\label{eq:newrelation}
\new^2=1,\quad \new^{-1} x\new=\new(x) \text{ for any }x\in \U_q(\gl(\mathfrak m|\mathfrak n)).
\end{equation}
As established in \cite{Ya94}, $\U$ is a Hopf algebra  whose comultiplication $\Delta$, counit $\epsilon$, antipode $S$ are given by
\begin{equation}
\label{eq:comultinew}
    \begin{aligned}
    &\Delta(q^h)=q^h\otimes q^h\quad \text{ for }h\in P^\vee, \\
    &\Delta(E_j)=E_j\otimes 1+\new^{p(j)}K_j\otimes E_j\quad \text{ for }j\in I,\\
    &\Delta(F_j)=F_j\otimes K_j^{-1}+\new^{p(j)}\otimes F_j\quad \text{ for }j\in I,\\
    &\Delta(\new)=\new \otimes \new,\\
    &\epsilon(\new)=\epsilon(q^h)=1, \text{ for }h\in P^\vee,\quad \epsilon(E_j)=\epsilon(F_j)=0,\quad \text{ for }j\in I, \\
    &S(\new)=\new,\quad S(q^{h})=q^{-h} \quad \text{ for }h\in P^\vee,\\
    &S(F_j)=-\new^{p(j)}F_jK_j,\quad S(E_j)=-\new^{p(j)}K_j^{-1}E_j\quad \text{ for }j\in I.
\end{aligned}
\end{equation}

We naturally extend the maps in \eqref{eq:rhoandop} from $\U_q(\gl(\mathfrak m|\mathfrak n))$ to $\U$ by setting
\[\sigma(\new)=\wp(\new)=\overline{\new}=\new.\]

\begin{remark}
The choice of the comultiplication we made here is different from \cite{Mi06}. When $\mathfrak g$ is a Lie algebra, our $\Delta$ is compatible with the comultipication in \cite{Lus93}.
\end{remark}

As in \cite{Lus93}, the multiplication map gives a triangular decomposition of $\U$:
\begin{equation}
\label{eq:triangle}
    \U\cong \U^+\otimes \U^0\otimes \U^-.
\end{equation}
where  $\U^{+}$ (resp. $\U^{-}$) denotes the subalgebra of $\U$ generated by $E_j $ (resp. $F_j $), $j\in I$ and $\U^{0}$ denotes the subalgebra of $\U$ generated by $\{q^\mu,\new \mid j\in I,\mu\in P^\vee\}$.

\section{Braid group operators}
\label{sec:BGO}
In order to define quantum supersymmetric pairs, we need to study the braid group operators on $\U$, especially the ones associated with odd simple roots. 

\subsection{Odd reflections}
As noted in \cite{Ya99}, when fixing a simple root $\alpha$, the braid operator associated with it extends the action of $s_{\alpha}$ on the weight data. The key distinction between odd and even reflections lies in the fact that odd reflections change the generalized Cartan matrix $A$, while even reflections do not; see also \cite[\S 4]{C16}.

The fundamental systems of the root system $\Phi$ associated to $\gl(\mathfrak m|\mathfrak n)$ are not conjugated under the Weyl group actions because of the existence of  odd roots (cf. \cite[\S 1.3.6]{CW12}). In fact, we have the following lemma,
\begin{lemma}
{\em \cite[Proposition 2.2.1]{Ya99}, \cite[Lemma 1.26]{CW12}}
Let $\alpha$ be an odd simple root of $\gl(\mathfrak m|\mathfrak n)$ in a positive system $\Phi^+$. Then,
$$\Phi_{\alpha}^+:=\{-\alpha\}\cup\Phi^+\backslash \{\alpha\}$$
is a new positive system, whose corresponding fundamental system $\Pi_{\alpha}$ is given by
\begin{equation}
    \Pi_{\alpha}=\{\beta\in \Pi\mid (\beta,\alpha)=0,\beta\neq\alpha\}\cup\{\beta+\alpha\mid\beta\in \Pi,(\beta,\alpha)\neq 0\}\cup \{-\alpha\}.
\end{equation}
\end{lemma}

The operation of obtaining $\Pi_\alpha$ from $\Pi$ is denoted by $s_{\alpha}$ and referred to as an odd reflection. When $\beta\in \Pi$ is an even simple root, we abuse the notation $s_{\beta}$ to denote the even reflection associated to $\beta$. For a diagram as in \eqref{eq:dynkin}, we let $s_j:=s_{\alpha_j}$ for all $j\in I$. 

Let $\mathcal D_{\mathfrak m,\mathfrak n}$ denote the set of all possible Dynkin diagrams for $\gl(\mathfrak m|\mathfrak n)$. The following lemma provides information on how the reflections change parities, which enables us to determine the matrix units of $A$. For any diagram $X\in \mathcal D_{\mathfrak m,\mathfrak n}$, we denote by $p_{\texttt{X}}$ the corresponding parity function.
\begin{lemma}
\label{eq:parity}
If $j,k,\ell \in I$ with $j\sim k$ and $j\nsim \ell$. Then for any $X\in \mathcal D_{\mathfrak m,\mathfrak n}$ we have
\begin{equation*}
    p_{s_j(X)}(j)=p_X(j),\quad p_{s_j(X)}(k)=p_X(k)+p_X(j) \mod 2,\quad
    p_{s_j(X)}(\ell)=p_X(\ell).
\end{equation*}
\end{lemma}

\begin{proof}
    We always have $s_j(\alpha_j)=-\alpha_j$, $s_j(\alpha_\ell)=\alpha_\ell$ and $s_j(\alpha_k)=\alpha_k+ \alpha_j$ for a diagram as in \eqref{eq:dynkin}.
\end{proof}

From Lemma~\ref{eq:parity} we see that even reflections will not change the parity of any simple root while odd reflections change the parities of the ones adjacent to it.

\begin{example}
\[
\begin{tikzpicture}[scale=1, semithick]
\node (1) [label=below:{$\epsilon_{\overline 1}-\epsilon_{\underline 1}$},scale=0.6] at (0,0){$\bigotimes$};
\node (2) [label=below:{$\epsilon_{\underline 1}-\epsilon_{\overline 2}$},scale=0.6] at (2,0){$\bigotimes$};
\node(3) at (3,0){$\overset{s_1}{\Longrightarrow}$};
\node (4) [label=below:{$\epsilon_{\underline 1}-\epsilon_{\overline 1}$},scale=0.6] at (4,0){$\bigotimes$};
\node (5) [circle,fill,draw,label=below:{$\epsilon_{\overline 1}-\epsilon_{\overline 2}$},scale=0.6] at (6,0){};
\path  
        (1) edge (2)
        (4) edge (5);
\end{tikzpicture}
\]
\end{example}
More precisely, for any two fundamental systems $\Pi$ and $\Pi'$ of a basic Lie superalgebra of any classical type, there exists a sequence consisting of even and odd relfections $s_1,\ldots,s_k$ such that $s_1\cdots s_k (\Pi)=\Pi'$. (cf. \cite{CW12})

\subsection{Braid group operators}
For each $X\in \mathcal D_{\mathfrak m,\mathfrak n}$, we can associate a quantum enveloping algebra $\U(X)$ with generators $E^{\texttt{X}}_{i},\ F^{\texttt{X}}_{i},\ q^\mu$ and $\new_{\texttt X}$ as in \S~\ref{sec:2.1}. Equipped with this family of algebras, the braid group operators were constructed in \cite[Proposition 7.5.1]{Ya99}. In \cite[Theorem 4.5]{C16}, an equivalent reformulation of these operators was given in the case of $\gl(\mathfrak m|1)$. In Theorem~\ref{thm:braid}, we adopt the notations from \cite[Theorem 4.5]{C16} and restate the results of \cite[Proposition 7.5.1]{Ya99} specifically for $\gl(\mathfrak m|\mathfrak n)$.

\begin{theorem}
\label{thm:braid}
Let $i\in I,\  X\in \mathcal D_{\mathfrak m,\mathfrak n},\ e=\pm 1$ and set $Y=s_i(X)$. There exist $\Q(q)$-linear algebra isomorphisms $T_{i,e}',T_{i,e}'':\U(X)\to \U(Y)$ satisfying
\begin{equation}
    T_{i,-e}'(E^{\texttt{X}}_{j})=
    \begin{cases}
    -(-1)^{p_Y(i)}K_{\texttt{Y},i}^{-e}F_{\texttt{Y},i},&\quad \text{if}\quad j=i, \\
    E_{\texttt{Y},j}E_{\texttt{Y},i}-(-1)^{p_Y(i)p_Y(j)}q^{e(\alpha_{\texttt{Y},i},\alpha_{\texttt{Y},j})}E_{\texttt{Y},i}E_{\texttt{Y},j} &\quad \text{if}\quad j \sim i,\\
    E_{\texttt{Y},j} &\quad \text{if}\quad j \nsim i.
    \end{cases}
\end{equation}
\begin{equation}
    T_{i,-e}'(F^{\texttt{X}}_{j})=
    \begin{cases}
    -(-1)^{p_Y(i)}E_{\texttt{Y},i}K_{Y,i}^{e},&\quad \text{if}\quad j=i, \\
    F_{\texttt{Y},i}F_{\texttt{Y},j}-(-1)^{p_Y(i)p_Y(j)}q^{-e(\alpha_{Y,i},\alpha_{Y,j})}F_{\texttt{Y},j}F_{\texttt{Y},i} &\quad \text{if}\quad j \sim i,\\
    F_{\texttt{Y},j} &\quad \text{if}\quad j \nsim i.
    \end{cases}
\end{equation}
\begin{equation}
    T_{i,-e}'(K^{\texttt{X}}_{j})=
    \begin{cases}
    (-1)^{p_Y(i)}K_{\texttt{Y},i}^{-1},&\quad \text{if}\quad j=i, \\
   (-1)^{p_Y(i)p_Y(j)}K_{\texttt{Y},i}K_{\texttt{Y},j} &\quad \text{if}\quad j \sim i,\\
    K_{\texttt{Y},j} &\quad \text{if}\quad j \nsim i.
    \end{cases}
\end{equation}
\begin{equation}
    T_{i,-e}'(\new_{\texttt{X}})=\new_{\texttt{Y}}.
\end{equation}
and
\begin{equation}
    T_{i,e}''(E^{\texttt{X}}_{j})=
    \begin{cases}
    -F_{\texttt{Y},i}K_{\texttt{Y},i}^{e},&\quad \text{if}\quad j=i, \\
    E_{\texttt{Y},i}E_{\texttt{Y},j}-(-1)^{p_Y(i)p_Y(j)}q^{e(\alpha_{\texttt{Y},i},\alpha_{\texttt{Y},j})}E_{\texttt{Y},j}E_{\texttt{Y},i} &\quad \text{if}\quad j \sim i,\\
    E_{\texttt{Y},j} &\quad \text{if}\quad j \nsim i.
    \end{cases}
\end{equation}
\begin{equation}
    T_{i,e}''(F^{\texttt{X}}_{j})=
    \begin{cases}
    -K_{\texttt{Y},i}^{-e}E_{\texttt{Y},i},&\quad \text{if}\quad j=i, \\
    F_{\texttt{Y},j}F_{\texttt{Y},i}-(-1)^{p_Y(i)p_Y(j)}q^{-e(\alpha_{\texttt{Y},i},\alpha_{\texttt{Y},j})}F_{\texttt{Y},i}F_{\texttt{Y},j} &\quad \text{if}\quad j \sim i,\\
    F_{\texttt{Y},j} &\quad \text{if}\quad j \nsim i.
    \end{cases}
\end{equation}
\begin{equation}
    T_{i,e}''(K_{\texttt{Y},j})=
    \begin{cases}
    (-1)^{p_Y(i)}K_{\texttt{Y},i}^{-1},&\quad \text{if}\quad j=i, \\
   (-1)^{p_Y(i)p_Y(j)}K_{\texttt{Y},i}K_{\texttt{Y},j} &\quad \text{if}\quad j \sim i,\\
    K_{\texttt{Y},j} &\quad \text{if}\quad j \nsim i.
    \end{cases}
\end{equation}
\begin{equation}
    T_{i,e}''(\new_{\texttt X})=\new_{\texttt Y}.
    \end{equation}

\end{theorem}

To ensure self-consistency, we will now present the proof of Theorem \ref{thm:braid} in the remaining part of this subsection. To do so succinctly, recall $\sigma$ and $\overline{\cdot}$ from \eqref{eq:rhoandop}, we observe that
\begin{equation}
    T_{i,-e}'=\sigma T_{i,e}''\sigma,\quad T'_{i,-e}=\overline{\cdot}T_{i,e}'\overline{\cdot}, \quad T_{i,e}''=(T_{i,-e}')^{-1}.
\end{equation}
One can check these identities on the generators of $\U(X)$. Thus to establish Theorem \ref{thm:braid}, it is sufficient to focus on the case of $T_{j,-1}'$. Specifically, we need to demonstrate that the images of the generators of $\U(X)$ under $T_{j,-1}'$ satisfy the relations in \eqref{eq:Urelation} and \eqref{eq:newrelation}. To ensure the clarity of the proof, we will break down the verification into lemmas and make reference to relevant results from \cite{C16}. Given the complexity of the calculations involved, we will omit the subscripts on the generators of $\U(Y)$ for readability. Additionally, we will consistently omit the subscript on $\new$ since it is evident from the context which algebra it belongs to.

First we take a look at \eqref{eq:newrelation}, we have
\begin{lemma}
If $j,k\in I$, then 
\begin{align*}
&\new T'_{j,-1}(E^{\texttt X}_{k})\new^{-1}=T'_{j,-1}(\new(E^{\texttt X}_{k}))=(-1)^{p_X(k)}T'_{j,-1}(E^{\texttt X}_{k}),\\
&\new T'_{j,-1}(F^{\texttt X}_{k})\new^{-1}=T'_{j,-1}(\new(F^{\texttt X}_{k}))=(-1)^{p_X(k)}T'_{j,-1}(F^{\texttt X}_{k}),\\
&\new T'_{j,-1}(K^{\texttt X}_{k})\new^{-1}=T'_{j,-1}(\new(K^{\texttt X}_{k}))=(-1)^{p_X(k)}T'_{j,-1}(K^{\texttt X}_{k}).
\end{align*}
\end{lemma}

\begin{proof}
We prove for $E$ and the other two are similar. 

When $j=k$ or $j\nsim k$, by a direct computation we have $p(k)=p_X(k)$ and  
\[\new T'_{j,-1}(E^{\texttt X}_{k})\new^{-1}=(-1)^{p(k)}T'_{j,-1}(E^{\texttt X}_{k})=(-1)^{p_X(k)}T'_{j,-1}(E^{\texttt X}_{k}).\]

When $j\sim k$, by a direct computation we have $p(k)=p_X(j)+p_X(k)$ and thus $p(k)+p(j)=p_X(k)$. Hence
\[\new T'_{j,-1}(E^{\texttt X}_{k})\new^{-1}=(-1)^{p(k)+p(j)}T'_{j,-1}(E^{\texttt X}_{k})=(-1)^{p_X(k)}T'_{j,-1}(E^{\texttt X}_{k}).\]
This proves the lemma.
\end{proof}

Recall the defining relations of $\U$ from \eqref{eq:Urelation}. The relations (R1)--(R4) can be verified directly. For the relation (R6), we have 
\begin{lemma}
\label{lemma:square}
If $j,k\in I$ such that $p_X(k)=1$, then
$$T_{j,-1}'(E^{\texttt{X}}_{k})^2=T_{j,-1}'(F^{\texttt{X}}_{k})^2=0.$$
\end{lemma}
\begin{proof}
    It follows from the same argument as in \cite[Lemma 4.7]{C16}.
\end{proof}

To verify the relation (R5), we split into two cases.
\begin{lemma}
If $j=k\in I$, then
\begin{align*}
    &T_{j,-1}'(E^{\texttt{X}}_{k})T_{j,-1}'(F^{\texttt{X}}_{k})-(-1)^{p_X(k)}T_{j,-1}'(F^{\texttt{X}}_{k})T_{j,-1}'(E^{\texttt{X}}_{k})\\
    =&\frac{T_{j,-1}'(K^{\texttt{X}}_{k})-T_{j,-1}'(K^{\texttt{X}}_{k})^{-1}}{q^{\ell_k}-q^{-\ell_k}}.
\end{align*}
\end{lemma}
\begin{proof}
    It follows from the same argument as in \cite[Lemma 4.9]{C16}.
\end{proof}

To verify the relation (R5) when $k\neq l$, note that we need the following lemma. 
\begin{lemma}
If $k=j-1$ and $\ell=j+1$, then we have
\begin{equation}
\label{eq:parityjkl}
    p_X(k)p_X(\ell)+p(k)p(\ell)+p(j)p(\ell)+p(j)p(k)\equiv p(j) \mod 2.
\end{equation}
\end{lemma}
\begin{proof}
This lemma follows from \eqref{eq:parity} and a direct computation.
\end{proof}

\begin{lemma}{\em (Compare \cite[Lemma 4.8]{C16})}
\label{lemma:commutator}
If $j,k,\ell\in I$ with $k\neq \ell$, then
$$T_{j,-1}'(E^{\texttt{X}}_{k})T_{j,-1}'(F^{\texttt{X}}_{\ell})=(-1)^{p_X(k)p_X(\ell)}T_{j,-1}'(F^{\texttt{X}}_{\ell})T_{j,-1}'(E^{\texttt{X}}_{k})$$
\end{lemma}
\begin{proof}
Let $c_{k,\ell}=T_{j,-1}'(E^{\texttt{X}}_{k})T_{j,-1}'(F^{\texttt{X}}_{\ell})-(-1)^{p_X(k)p_X(\ell)}T_{j,-1}'(F^{\texttt{X}}_{\ell})T_{j,-1}'(E^{\texttt{X}}_{k})$. We want to prove $c_{k,\ell}=0$ for all $k\neq \ell$. 

If one of them is not connected to $j$, let us say $j\nsim k$, then $T_{j,-1}'(E^{\texttt{X}}_{k})=E_k$ and $p(j)=p_X(j)$. On the other hand, $T_{j,-1}'(E^{\texttt{X}}_{\ell})$ is a polynomial in the elements $K_j,F_j,F_\ell$ and $F_j$ with $p(T_{j,-1}'(F^{\texttt{X}}_{\ell}))=p_X(\ell)$. Since $E_k$ super-commutes with all of those elements, the statement follows.

The remaining cases involve situations where both $k$ and $\ell$ are either equal to or connected with $j$. For the case where one of them is connected to $j$ and the other is equal to $j$, the verification has already been conducted in \cite[Lemma 4.8]{C16}.
 
When $k$ and $\ell$ are both connected to $j$, without loss of generality, we assume that $k=j-1$ and $\ell=j+1$. Using \eqref{eq:parityjkl} and the relation (R5) in \eqref{eq:Urelation} repeatedly we get
\begin{align*}
    c_{k,\ell}=&(E_{k}E_{j}-(-1)^{p(j)p(k)}q^{(\alpha_j,\alpha_k)}E_{j}E_{k})(F_{j}F_{\ell}-(-1)^{p(j)p(\ell)}q^{-(\alpha_j,\alpha_\ell)}F_{\ell}F_{j}) \\
    &-(-1)^{p_X(k)p_X(\ell)}(F_{j}F_{\ell}-(-1)^{p(j)p(\ell)}q^{-(\alpha_j,\alpha_\ell)}F_{\ell}F_{j})\\
    &\qquad(E_{k}E_{j}-(-1)^{p(j)p(k)}q^{(\alpha_j,\alpha_k)}E_{j}E_{k}) \\
    =&E_k[E_j,F_j]F_\ell-q^{(\alpha_j,\alpha_k)}[E_j,F_j]E_kF_\ell-q^{-(\alpha_j,\alpha_\ell)}E_kF_\ell[E_j,F_j]\\
    &+(-1)^{p(k)p(\ell)}q^{(\alpha_j,\alpha_k)-(\alpha_j,\alpha_\ell)}F_\ell[E_j,F_j]E_k\\
    =&\frac{1}{q^{\ell_j}-q^{-\ell_j}}E_k((1-q^{2(\alpha_j,\alpha_k)}-1+q^{2(\alpha_j,\alpha_k)})K_j\\
    &\qquad-(1-1-q^{-2(\alpha_j,\alpha_\ell)}+q^{-2(\alpha_j,\alpha_\ell)})K_j^{-1})F_\ell=0.
\end{align*}
This proves the lemma.
\end{proof}


The next lemma checks the relation (R7). 
\begin{lemma}
If $j,k,\ell\in I$ such that $k\nsim \ell$, then
$$T_{j,-1}'(E^{\texttt{X}}_{k})T_{j,-1}'(E^{\texttt{X}}_{\ell})=(-1)^{p_X(k)p_X(\ell)}T_{j,-1}'(E_{\texttt{X},
\ell})T_{j,-1}'(E^{\texttt{X}}_{k}),$$
$$T_{j,-1}'(F^{\texttt{X}}_{k})T_{j,-1}'(F^{\texttt{X}}_{\ell})=(-1)^{p_X(k)p_X(\ell)}T_{j,-1}'(F_{\texttt{X},
\ell})T_{j,-1}'(F^{\texttt{X}}_{k}).$$
\end{lemma}
\begin{proof}
We only prove for $E$. If either $k$ or $\ell$ is not connected to $j$, we are done. So we suppose that $k=j-1$ and $\ell=j+1$. We break the proof into two cases:

(Case-1) Assume $p_X(j)=0$. Observe that in this case we must have $(\alpha_{j},\alpha_{k})=(\alpha_{j},\alpha_{\ell})$. Without loss of generality we assume $(\alpha_{j},\alpha_{k})=(\alpha_{j},\alpha_{\ell})=1$. Then we have
$T_{j,-1}'(E^{\texttt{X}}_{k})=E_{k}E_{j}-qE_{j}E_{k}$ and $T_{j,-1}'(E^{\texttt{X}}_{\ell})=E_{\ell}E_{j}-qE_{j}E_{\ell}$ and thus
\begin{align*}
    &T_{j,-1}'(E^{\texttt{X}}_{k})T_{j,-1}'(E^{\texttt{X}}_{\ell})=E_kE_jE_\ell E_j-qE_jE_kE_\ell E_j-qE_kE^2_jE_\ell+q^2E_jE_kE_jE_\ell, \\
    &T_{j,-1}'(E^{\texttt{X}}_{\ell})T_{j,-1}'(E^{\texttt{X}}_{k})=E_\ell E_jE_k E_j-qE_jE_\ell E_k E_j-qE_\ell E^2_jE_k+q^2E_jE_\ell E_jE_k.
\end{align*}
First we see that $E_jE_kE_\ell E_j=(-1)^{p(k)p(\ell)}E_jE_\ell E_k E_j=(-1)^{p_X(k)p_X(\ell)}E_jE_\ell E_k E_j$. Thus, by applying (R8) repeatedly we get
\begin{align*}
    &T_{j,-1}'(E^{\texttt{X}}_{k})T_{j,-1}'(E^{\texttt{X}}_{\ell})-(-1)^{p_X(k)p_X(\ell)}T_{j,-1}'(E_{\texttt{X},
\ell})T_{j,-1}'(E^{\texttt{X}}_{k}) \\
=&E_kE_jE_\ell E_j-q(E_kE^2_j-qE_jE_kE_j)E_\ell 
\\
&\qquad-(-1)^{p_X(k)p_X(\ell)}[E_\ell E_jE_k E_j-q(E_\ell E^2_j-qE_jE_\ell E_j)E_k]\\
=&E_kE_jE_\ell E_j-q(q^{-1}E_jE_kE_j-E^2_jE_k)E_\ell\\
&\quad-(-1)^{p_X(k)p_X(\ell)}[E_\ell E_jE_k E_j-q(q^{-1}E_jE_\ell E_j-E^2_jE_\ell)E_k]\\
=&\frac{1}{q+q^{-1}}[E_k(E_j^2E_\ell+E_\ell E_j^2)-(E_kE_j^2+E_j^2E_k)E_\ell\\
&\quad-(-1)^{p_X(k)p_X(\ell)}E_\ell(E_kE_j^2+E_j^2E_k)+(-1)^{p_X(k)p_X(\ell)}(E_j^2E_\ell+E_\ell E_j^2)E_k]\\
=&0.
\end{align*}

(Case-2) Assume $p_X(j)=1$. In this case we always have $(\alpha_{j},\alpha_{k})=-(\alpha_{j},\alpha_{\ell})$. Again we may assume that $(\alpha_{j},\alpha_{k})=1$. Then $(\alpha_j,\alpha_\ell)=-1$. Thus we have $T_{j,-1}'(E^{\texttt{X}}_{k})=E_{k}E_{j}-(-1)^{p(k)}qE_{j}E_{k}$ and $T_{j,-1}'(E^{\texttt{X}}_{\ell})=E_{\ell}E_{j}-(-1)^{p(\ell)}q^{-1}E_{j}E_{\ell}$ and thus
\begin{align*}
    &T_{j,-1}'(E^{\texttt{X}}_{k})T_{j,-1}'(E^{\texttt{X}}_{\ell})=E_kE_jE_\ell E_j-(-1)^{p(k)}qE_jE_kE_\ell E_j+(-1)^{p(k)+p(\ell)}E_jE_kE_jE_\ell, \\
    &T_{j,-1}'(E^{\texttt{X}}_{\ell})T_{j,-1}'(E^{\texttt{X}}_{k})=E_\ell E_jE_k E_j-(-1)^{p(\ell)}q^{-1}E_jE_\ell E_k E_j+(-1)^{p(k)+p(\ell)}E_jE_\ell E_jE_k.
\end{align*}
By taking a difference of the above two equations and unravelling the relation (R10) we can conclude that $$T_{j,-1}'(E^{\texttt{X}}_{k})T_{j,-1}'(E^{\texttt{X}}_{\ell})=(-1)^{p_X(k)p_X(\ell)}T_{j,-1}'(E^{\texttt{X}}_{\ell})T_{j,-1}'(E^{\texttt{X}}_{k}).$$
This proves the lemma.
\end{proof}

The verification process for the relations (R8) and (R9) is no different from that in the $\gl(m|1)$ case. Hence we have
\begin{lemma}{\em\cite[Lemma 4.11]{C16}}
If $j,k,\ell\in I$ such that $p_X(k)=0$ and $k\sim \ell$, then
\begin{align*}
    &T_{j,-1}'(E^{\texttt{X}}_{k})^2T_{j,-1}'(E^{\texttt{X}}_{\ell})-(q+q^{-1})T_{j,-1}'(E^{\texttt{X}}_{k})T_{j,-1}'(E^{\texttt{X}}_{\ell})T_{j,-1}'(E^{\texttt{X}}_{k})\\
    &\qquad+T_{j,-1}'(E^{\texttt{X}}_{\ell})T_{j,-1}'(E^{\texttt{X}}_{k})^2=0,\\
    &T_{j,-1}'(F^{\texttt{X}}_{k})^2T_{j,-1}'(F^{\texttt{X}}_{\ell})-(q+q^{-1})T_{j,-1}'(F^{\texttt{X}}_{k})T_{j,-1}'(F^{\texttt{X}}_{\ell})T_{j,-1}'(F^{\texttt{X}}_{k})\\
    &\qquad+T_{j,-1}'(F^{\texttt{X}}_{\ell})T_{j,-1}'(F^{\texttt{X}}_{k})^2=0.
\end{align*}
\end{lemma}

Finally we need to verify the relations (R10) and (R11).
\begin{lemma} Let $z,k,j,\ell\in I$ with $k\sim j\sim \ell,\ k<\ell$ and $p_X(j)=1$, then
\begin{align*}
    &S_{p_X(k),p_X(\ell)}(T_{z,-1}'(E^{\texttt{X}}_{k}),T_{z,-1}'(E^{\texttt{X}}_{j}),T_{z,-1}'(E^{\texttt{X}}_{\ell}))=0,\\
    &S_{p_X(k),p_X(\ell)}(T_{z,-1}'(F^{\texttt{X}}_{k}),T_{z,-1}'(F^{\texttt{X}}_{j}),T_{z,-1}'(F^{\texttt{X}}_{\ell}))=0.
\end{align*}
\end{lemma}

\begin{proof}
We only prove the first equality as the second one can be proved similarly. If none of $k,j,\ell$ is connected or equal to $z$, there is nothing to prove.

  Now we first suppose that $j\nsim z$ and $k\sim z$. In this case we have $T_{z,-1}'(E^{\texttt{X}}_{k})=E_kE_z-(-1)^{p(z)p(k)}q^{(\alpha_k,\alpha_z)}E_zE_k,\ T_{z,-1}'(E^{\texttt{X}}_{j})=E_j,\ T_{z,-1}'(E^{\texttt{X}}_{\ell})=E_\ell$. Thus
\begin{align*}
    &S_{p_X(k),p_X(\ell)}(T_{z,-1}'(E^{\texttt{X}}_{k}),T_{z,-1}'(E^{\texttt{X}}_{j}),T_{z,-1}'(E^{\texttt{X}}_{\ell})) \\
    =&S_{p_X(k),p_X(\ell)}(E_kE_z,E_j,E_\ell)-(-1)^{p(z)p(k)}q^{(\alpha_z,\alpha_k)}S_{p_X(k),p_X(\ell)}(E_zE_k,E_j,E_\ell)
\end{align*}
 
When $p(z)=0$, $E_z$ commutes with $E_j$ and $E_\ell$. Also we have $p_X(k)=p(k),\ p_X(\ell)=p(\ell)$. Hence
\begin{align*}
    S_{p_X(k),p_X(\ell)}(E_kE_z,E_j,E_\ell)=S_{p(k),p(\ell)}(E_k,E_j,E_\ell)E_z=0,\\
    S_{p_X(k),p_X(\ell)}(E_zE_k,E_j,E_\ell)=E_zS_{p_X(k),p_X(\ell)}(E_k,E_j,E_\ell)=0.
\end{align*}

When $p(z)=1$, since $E_zE_j=(-1)^{p(z)}E_jE_z$, $ E_zE_\ell=(-1)^{p(z)p(\ell)}E_\ell E_z$, $p(k)=p_X(k)+1$ and $p(\ell)=p_X(\ell)$, again we have
\begin{align*}
    S_{p_X(k),p_X(\ell)}(E_kE_z,E_j,E_\ell)=S_{p(k),p(\ell)}(E_k,E_j,E_\ell)E_z=0,\\
    S_{p_X(k),p_X(\ell)}(E_zE_k,E_j,E_\ell)=E_zS_{p_X(k),p_X(\ell)}(E_k,E_j,E_\ell)=0.
\end{align*}
Note that the case when $j\nsim h,\ \ell \sim h$ is similar.

Next, suppose $z\sim j$ and without loss of generality that $z=k$. We further assume that $(\alpha_j,\alpha_k)=-1$, thus $(\alpha_j,\alpha_\ell)=-1$. 
Note that when $p(k)p(\ell)=0$, the proof is already given in \cite[Lemma 4.12]{C16}. So we only need to consider the case when
$p(z)=p(k)=p(\ell)=1$. In this case we have $p_X(k)=p_X(\ell)=1$, $p(j)=0$ and $T_{z,-1}'(E^{\texttt{X}}_{j})=E_jE_z-q^{-1}E_zE_j,\ T_{z,-1}'(E^{\texttt{X}}_{\ell})=E^{\texttt{X}}_{\ell}$ and $T_{z,-1}'(E^{\texttt{X}}_{k})=K_k^{-1}F_k.$ Let \[\eth_{abcd}=T_{z,-1}'(E^{\texttt{X}}_{a})T_{z,-1}'(E^{\texttt{X}}_{b})T_{z,-1}'(E^{\texttt{X}}_{c})T_{z,-1}'(E^{\texttt{X}}_{d}).\] The goal is to prove that $$(q+q^{-1})\eth_{j\ell kj}=-\eth_{j\ell jk}+\eth_{kj\ell j}+\eth_{jkj\ell}-\eth_{\ell jkj}.$$
Note the identities
\begin{align*}
   &T_{z,-1}'(E^{\texttt{X}}_{k})T_{z,-1}'(E^{\texttt{X}}_{\ell})=-T_{z,-1}'(E^{\texttt{X}}_{\ell})T_{z,-1}'(E^{\texttt{X}}_{k}),\\
   &T_{z,-1}'(E^{\texttt{X}}_{j})T_{z,-1}'(E^{\texttt{X}}_{k})=-q^{-1}T_{z,-1}'(E^{\texttt{X}}_{k})T_{z,-1}'(E^{\texttt{X}}_{j})+q^{-1}E_j.
\end{align*}
With the above identities and Lemma \ref{lemma:square} we see that
\begin{align*}
    &\eth_{j\ell jk}=-q^{-1}\eth_{j\ell kj}+q^{-1}E_jE_kE_\ell E_j-q^{-2}E_kE_jE_\ell E_j,\\
    &\eth_{kj\ell j}=-q\eth_{jk\ell j}+E_jE_\ell E_jE_k-q^{-1}E_jE_kE_\ell  E_j,\\
    &\eth_{jkj\ell}=E_jE_kE_jE_\ell,\\
    &\eth_{\ell jkj}=E_\ell E_jE_kE_j
\end{align*}
Thus using Serre relation (R8) repeatedly we have
$$(q+q^{-1})\eth_{j\ell kj}=-\eth_{j\ell jk}+\eth_{kj\ell j}+\eth_{jkj\ell}-\eth_{\ell jkj}.$$

Finally we suppose that $z=j$. Again when $p(k)p(\ell)=0$ the proof is given in \cite[Lemma 4.12]{C16}. So we only need to consider the case when $p(k)=p(\ell)=1$. Thus we have $p_X(k)=p_X(\ell)=0$. Without loss of generality we assume that $(\alpha_k,\alpha_j)=-(\alpha_j,\alpha_\ell)=-1$. Then $T_{z,-1}'(E^{\texttt{X}}_{k})=E_kE_j+q^{-1}E_jE_k$, $T_{z,-1}'(E^{\texttt{X}}_{\ell})=E_\ell E_j+qE_j E_\ell$ and $T_{z,-1}'(E^{\texttt{X}}_{j})=K_j^{-1}F_j$. Note that we have the identities
\begin{align*}
    &T_{j,-1}'(E^{\texttt{X}}_{k})T_{j,-1}'(E^{\texttt{X}}_{j})=q^{-1}T_{j,-1}'(E^{\texttt{X}}_{j})T_{j,-1}'(E^{\texttt{X}}_{k})+q^{-1}E_k, \\
    &T_{j,-1}'(E^{\texttt{X}}_{\ell})T_{j,-1}'(E^{\texttt{X}}_{j})=qT_{j,-1}'(E^{\texttt{X}}_{j})T_{j,-1}'(E^{\texttt{X}}_{\ell})-qE_\ell.
\end{align*}
Thus we have
\begin{align*}
    &\eth_{j\ell jk}=-K_j^{-1}F_j(qE_\ell E_kE_j+E_\ell E_j E_k),\\
    & \eth_{kj\ell j}=-E_kE_\ell-K_j^{-1}F_j(E_kE_jE_\ell+q^{-1}E_jE_kE_\ell),\\
    & \eth_{jkj\ell}=K_j^{-1}F_j(q^{-1}E_kE_\ell E_j+E_kE_jE_\ell),\\
    & \eth_{\ell jkj}=-E_\ell E_k+K_j^{-1}F_j(E_\ell E_jE_k+qE_jE_\ell E_k),\\
    &\eth_{jk\ell j}=\eth_{j\ell kj}=K_jF_j^{-1}(E_kE_\ell E_j-E_jE_kE_\ell).
\end{align*}
Then we conclude that
$$(q+q^{-1})\eth_{j\ell kj}=\eth_{j\ell jk}+\eth_{kj\ell j}+\eth_{jkj\ell}+\eth_{\ell jkj}.$$
This proves the lemma.
\end{proof}

We have now proved that $T_{j,e}'$ and $T_{j,e}''$ are algebra isomorphisms for all $j\in I$ and $e=\pm 1$. The next proposition states that the braid group operators in Theorem~\ref{thm:braid} satisfy the type A braid relations.

\begin{proposition}
\label{prop:braid}
Let $j,k,\ell\in I$ and $X\in \mathcal D_{\mathfrak m,\mathfrak n}$.

(1) If $j\nsim k$, then $T_{j,e}'T_{k,e}'=T_{k,e}'T_{j,e}'$ and $T_{j,e}''T_{k,e}''=T_{k,e}''T_{j,e}''$.

 (2) If $j\sim k$ and $Y=s_js_k(X)$, then
\begin{align*}
    T_{j,-e}'T_{k,-e}'(E^{\texttt{X}}_{j})=T_{j,e}''T_{k,e}''(E^{\texttt{X}}_{j})=E^{\texttt{Y}}_{k}, \\
    T_{j,-e}'T_{k,-e}'(F^{\texttt{X}}_{j})=T_{j,e}''T_{k,e}''(F^{\texttt{X}}_{j})=F^{\texttt{Y}}_{k}, \\
    T_{j,-e}'T_{k,-e}'(K^{\texttt{X}}_{j})=T_{j,e}''T_{k,e}''(K^{\texttt{X}}_{j})=K^{\texttt{Y}}_{k}.
\end{align*}

(3) If $j\sim k$, then
$T_{j,e}'T_{k,e}'T_{j,e}'=T_{k,e}'T_{j,e}'T_{k,e}'$ and $T_{j,e}''T_{k,e}''T_{j,e}''=T_{k,e}''T_{j,e}''T_{k,e}''$.
\end{proposition}

\begin{proof}
    It follows from \cite[Lemma 8.1.1]{Ya99}; see also \cite[\S 6.3]{H10}.
\end{proof}

 From now on we denote by $T_i$ the braid operator $T_{i,1}''$ defined in Theorem \ref{thm:braid}. The next lemma can be proved similarly as in \cite[\S 8.18--\S 8.20]{Jan95}.

\begin{lemma}
\label{lemma:walpha=beta}
Let $w\in W$, $X\in \mathcal D_{\mathfrak m,\mathfrak n}$, $Y=w(X)$ and $\alpha\in \Pi_X$.  If $w(\alpha)>0$ in the root system associated to $X$, then $T_{w}(E^{\texttt{X}}_{\alpha})\in \U(Y)^+$. If $w(\alpha)\in \Pi_X$, then $T_{w}(E^{\texttt{X}}_{\alpha})=E^{\texttt{Y}}_{w(\alpha)}$.
\end{lemma}



\section{Quantum supersymmetric pair of type AIII}
\label{sec:QSP}
In this section we define the quantum supersymmetric pairs and the corresponding $\io$quantum supergroups of type AIII.

\subsection{$\io$Quantum supergroup of type AIII} 
\label{sec:notation}
For a real number $x\in \R$ and $m \in \N$, we denote $[x, x+m] =\{x, x+1, \ldots, x+m \}$. 
For $a \in \Z_{\ge 1}$, we denote by 
\[
\I_a = \left [\frac{1-a}2, \frac{a-1}2 \right].
\]

Fix
\[
n =\frac{m}2 \in \frac12 \N. 
\]
We consider the Satake diagram of type AIII with $\nb-1 =2\pt-1$ black nodes and $r$ pairs of white nodes, together with a diagram involution $\tau$ indicated by the dashed arrows:
\begin{equation}
\label{eq:AIIIdiagram}
\begin{tikzpicture}[scale=1, semithick]
\node (-4) [circle,draw,label=above:{$-\pt-r+1$},scale=0.6] at (0,0){$\cdot$};
\node (-3)  at (2,0) {$\cdots$} ;
\node (-2) [circle,draw,label=above:{$-\pt$},scale=0.6] at (4,0){$\cdot$};
\node (-1) [rectangle,fill,draw,label=above:{$-\pt+1$},scale=0.6] at (5,-0.5){};
\node (0)  at (5,-1.5){$\vdots$};
\node (1) [rectangle,fill,draw,label=below:{$\pt-1$},scale=0.6] at (5,-2.5){};
\node (2) [circle,draw,label=below:{$\pt$},scale=0.6] at (4,-3){$\cdot$};
\node (3)  at (2,-3){$\cdots$};
\node (4) [circle,draw,label=below:{$\pt+r-1$},scale=0.6] at (0,-3){$\cdot$};
\path (-4) edge (-3)
          (-3) edge (-2)
          (-2) edge (-1)
          (-1) edge (0)
          (0) edge (1)
          (1) edge (2)
          (2) edge (3)
          (3) edge (4);
\path (-4) edge[dashed,bend right,<->] (4)
    (-2) edge[dashed,bend right,<->] (2)
    (-3) edge[dashed,bend right,<->] (3);
\end{tikzpicture}
\end{equation}
where $\bigodot$ stands for white dots and $\blacksquare$ stands for black dots. 
We will denote the white even roots, black even roots,  black odd roots and white odd roots respectively by $\emptycirc,\  \fullcirc,\ \halfcirc$ and $\bigotimes$.

Both white and black dots allow different parities under the following assumption:
\begin{equation}
\label{eq:assumption}
\begin{aligned}
&\#\{p(j)=1\mid j\in I_\bu\}\equiv 0 \mod 2, \\
&p(j)=p(\tau(j)), \quad \forall i\in I_\circ,\quad\\
&i\in \Ieven \text{ if }\tau i=i \text{ and }i\in I_\circ.
\end{aligned}
\end{equation}
 where $
I_\bu =[1-\pt, \pt-1],\ I_\circ=I\backslash I_\bu.
$ (In case $\pt=0$, the black nodes are dropped; the nodes $\pt$ and $-\pt$ are identified and fixed by $\tau$.) 

For any Satake diagram in $ \mathcal D_{\mathfrak m,\mathfrak n}$ of the form \eqref{eq:AIIIdiagram}, we denote the index set by \begin{equation}
\label{eq:Iindex}
    I=\I_{\nb+2r-1}=I_\circ\cup I_\bu,\qquad (\nb+2r=\mathfrak m+\mathfrak n).
\end{equation}
Switching to this notation has the advantage of easily identifying the diagram involution $\tau$ with $-1$ on the index set of the simple roots.

Let $\mathfrak{S}_{m-1}$ denote the symmetric group associated with $I_\bu=[1-n,n-1]$, and let $w_\bu$ represent the longest element of $\mathfrak{S}_{m-1}$. For any reduced expression $w_\bu= s_{i_1}\cdots s_{i_\ell}$ as a product of simple generators, we regard $s_{i_t}$ as the simple reflection $s_{\alpha_{i_t}}$. Consequently, we can view $w_{\bu}$ as a product of even and odd reflections. It follows from \cite{HY08} that $w_\bu$ is independent of the choice of the reduced expression.

 Following \cite{BW18b}, we define 
\begin{align}
\label{eq:L}
\begin{split}
\wl &= P\big /\{\mu+ w_\bu \tau (\mu) \mid \mu \in P \},
\\
\cwl &=  \{ \nu- w_\bu \tau (\nu) \mid \nu \in P^\vee\}.
\end{split}
\end{align}
We call an element in $\wl$ an $\imath$-weight and $\wl$ the $\imath$-weight lattice.

For any Satake diagram $X$ in the form of \eqref{eq:AIIIdiagram}, without considering the diagram involution $\tau$, the diagram $X$ corresponds to a Lie superalgebra $\gl(\mathfrak m|\mathfrak n)$ for certain non-negative integers $\mathfrak m$ and $\mathfrak n$, where $\mathfrak m+\mathfrak n=2r+\nb$. Recall $I(\mathfrak m|\mathfrak n)$ from \eqref{eq:Imn}. The simple roots of $X$ are given by $$\Pi_{\texttt X}=\{\alpha_{\texttt {X},k}=\epsilon^{\texttt{X}}_{k-\frac{1}{2}}-\epsilon^{\texttt{X}}_{k+\frac{1}{2}}\mid k\in I\}$$
where $\{ \epsilon^{\texttt X}_{k\pm \frac{1}{2}}\mid k\in I\}=\{\epsilon_a\mid a\in I(\mathfrak m|\mathfrak n)\}.$

In the remaining part of this section, we fix a diagram $X\in \mathcal D_{\mathfrak m,\mathfrak n}$ of the form \eqref{eq:AIIIdiagram} satisfying \eqref{eq:assumption}. Furthermore, we recall the definition of $\ell_j$ from equation \eqref{eq:lj}. In addition, we provide two lemmas that will be useful for future reference.
\begin{lemma}
\label{lemma:ljl-j}
 We have \[\ell_j=(-1)^{p(j)}\ell_{-j},\ \forall j\in I.\]
\end{lemma}

\begin{proof}
By observation we have $\ell_j=(-1)^{p\left(\epsilon^{\texttt{X}
}_{j-\frac{1}{2}}\right)}$. Moreover, since $p(j)=p(-j)$ for all $j\in I$, we have $\ell_j=(-1)^{p(j)}\ell_{-j},\ \forall j\in I.$
\end{proof}

\begin{lemma}
Suppose that $Y=w_\bu(X)$, then $Y \in \mathcal D_{\mathfrak m,\mathfrak n}$ and satisfies \eqref{eq:assumption}.
\end{lemma}

\begin{proof}
For each $k\in I$, we have\begin{equation}
\label{eq:varpi}
    \alpha_{\texttt{Y},k}:=w_\bu(\alpha_{\texttt{X},k})=\begin{cases}
    \epsilon^{\texttt X}_{k-\frac{1}{2}}-\epsilon^{\texttt{X}}_{k+\frac{1}{2}} &\text{ if } |k|>\pt, \\
    \epsilon^{\texttt{X}}_{-\pt+\frac{1}{2}}-\epsilon^{\texttt{X}}_{\pt+\frac{1}{2}}&\text{ if } k=\pt, \\
    \epsilon^{\texttt{X}}_{-k+\frac{1}{2}}-\epsilon^{\texttt{X}}_{-k-\frac{1}{2}}&\text{ if } -\pt<k<\pt, \\
    \epsilon^{\texttt{X}}_{-\pt-\frac{1}{2}}-\epsilon^{\texttt{X}}_{\pt-\frac{1}{2}}&\text{ if } k=-\pt.
    \end{cases}
\end{equation}
From \eqref{eq:varpi} we can see that $\alpha_{\texttt{X},k}=\alpha_{\texttt{Y},k}$ if $|k|>\pt$ and so is the parity. For $|k|=\pt$, suppose $p(\alpha_{\texttt{X},-\pt})=p(\alpha_{\texttt{X},\pt})=0$, then $\epsilon^{\texttt{X}}_{-\pt-\frac{1}{2}}$ and $\epsilon^{\texttt{X}}_{-\pt+\frac{1}{2}}$ have the same parity while $\epsilon^{\texttt{X}}_{\pt-\frac{1}{2}}$ and $\epsilon^{\texttt{X}}_{\pt+\frac{1}{2}}$ have the same parity. Thus $p(\alpha_{\texttt{Y},\pt})=p(\alpha_{\texttt{Y},-\pt})$. It can be checked similarly that when $p(\alpha_{\texttt{X},-\pt})=p(\alpha_{\texttt{X},\pt})=1$, we still have $p(\alpha_{\texttt{Y},\pt})=p(\alpha_{\texttt{Y},-\pt})$.

For $-\pt<k<\pt$, we see that $\alpha_{\texttt{Y},k}=-\alpha_{\texttt{X},-k}$. Thus the number of black odd roots stays unchanged. Moreover, if $\tau i=i$ and $i\in I_\circ$, then we have $I_\bu=\varnothing$. Hence $Y=X$.
\end{proof}
Let $Y:=w_\bu(X)$. According to \eqref{eq:varpi}, we see that $\alpha_{\texttt{Y},k}=\epsilon^{\texttt{Y}}_{k-\frac{1}{2}}-\epsilon^{\texttt{Y}}_{k+\frac{1}{2}}$ where
\begin{equation}
\label{eq:i_j}
    \epsilon^{\texttt{Y}}_{t}=\begin{cases}
   \epsilon^{\texttt{X}}_t, &\text{ if } t>\pt-\frac{1}{2} \text{ or } t\leqslant -\pt-\frac{1}{2},\\
     \epsilon^{\texttt{X}}_{-t}, & \text{ if } -n-\frac{1}{2}<t\leqslant n-\frac{1}{2}
    \end{cases}
\end{equation}

Let $\U(Y)$ represent the quantum supergroup associated with generators $\new, E^{\texttt{Y}}_{j}, F^{\texttt{Y}}_{j}, q^\mu$, where $j\in I$ and $\mu\in P^\vee$, corresponding to the Dynkin diagram $Y$. Similarly, let $\U(X)$ denote the same algebra with generators $\new, E^{\texttt{X}}_{j}, F^{\texttt{X}}_{j}, q^\mu$, where $j\in I$ and $\mu \in P^\vee$, but with a different presentation corresponding to the Dynkin diagram $X$. We note that the comultiplication $\Delta$ is dependent on the chosen presentation, as shown in equation \eqref{eq:comultinew}. For simplicity, we use the same notation $\Delta$ and the parity function $p$ for different presentations, and we omit the script $Y$ unless necessary. 

The $\imath$quantum supergroup of type AIII, denoted by $\Uinew=\Uinew(Y)$, is the $\Q(q)$-subalgebra of $\U(Y)$ generated by $q^\mu \ (\mu\in \cwl),\  E_j, F_j \ (j\in I_\bu)$, $\new$ and
\begin{align}
\label{eq:Bi}
 B_j= 
 F_j+\va_jT_{w_\bu}( E^{\texttt{X}}_{\tau j})  K_j^{-1},\ \ \text{for } j\in I_\circ.
\end{align}
where parameters $\va_j\in \Q(q)$, for $j\in I_\circ$ satisfy the conditions $\va_j=\va_{-j}$, for $j\in I_\circ \backslash \{\pm  n \}$ \cite{Let02} (also cf. \cite{BK15, BW21}).  
(When $n=0$, $ B_0$ will be allowed to take a more general form $ B_0 = F_0 +\va_0  E_0  K_0^{-1} +\kappa_0  K_0^{-1}$, for an additional parameter $\kappa_0 \in \Q(q)$.)

For each reduced expression $w_\bu=s_{j_1}\cdots s_{j_l}$, we can write $T_{w_\bu}=T_{j_1}\cdots T_{j_l}$. By Proposition \ref{prop:braid}, $T_{w_\bu}$ is a well-defined operator as a product of braid operators associated to both odd and even simple roots in $I_\bu$.

Now $(\U( Y), \Uinew(Y))$ forms a quantum supersymmetric pair of type AIII \cite{Let99, Let02} (cf. \cite{BW18a, BK19}). The algebra $\Uinew$ satisfies the following relations
\begin{align*}
&q^\mu  B_j = q^{-\la\mu,\alpha_j\ra}  B_jq^\mu,\ \forall j\in I_\circ,  \\
&q^\mu F_j = q^{-\la\mu,\alpha_j\ra}  F_jq^\mu, \quad
q^\mu  E_j = q^{\la\mu,\alpha_j\ra}  E_jq^\mu,\ \forall j\in I_\bu,  \mu \in \cwl, \\
&\new(B_j)=(-1)^{p(j)}B_j,\quad \forall j\in I.
\end{align*}
and additional Serre type relations. By definition we see the following relation holds in $\Uinew$.
\begin{lemma}
For any $j\in I_\bu,\ k\in I$ we have
\begin{equation}
\label{eq:EjBk}
E_jB_k-(-1)^{p(j)p(k)}B_kE_j=\delta_{jk}\frac{K_j-K_j^{-1}}{q^{\ell_j}-q^{-\ell_j}}.
\end{equation}
\end{lemma}

For future use, we let $\Ub$ denote the subalgebra of $\U$ generated by $\{E_j,F_j,K_j^{\pm 1},\new\mid j\in I_\bu\}$. Let $\Uio$ denote the subalgebra of $\Uinew$ generated by $\{q^\mu,\new \mid\mu\in \cwl\}$.

The next lemma will help us pin down one of the conditions on the parameters.
\begin{lemma}
\label{lemma:BjB-j}
If $\va_j\neq \va_{-j}$ for $j \in I_\circ\backslash \{\pm \pt\}$, then $(K_j^{-1}K_{-j}^{-1})\in \U^{\imath0}$.
\end{lemma}

\begin{proof}
This claim follows from the relation
\begin{align*}
    B_jB_{-j}-(-1)^{p(j)}B_{-j}B_j=\va_{-j}(-1)^{p(j)}\frac{K_j-K_j^{-1}}{q^{\ell_j}-q^{-\ell_j}}K^{-1}_{-j}+\va_j\frac{K_{-j}-K_{-j}^{-1}}{q^{\ell_{-j}}-q^{-\ell_{-j}}}K^{-1}_{j}
\end{align*}
According to Lemma \ref{lemma:ljl-j} we have $\ell_j=(-1)^{p(j)}\ell_{-j}$. Thus $q^{\ell_{-j}}-q^{-\ell_{-j}}=(-1)^{p(j)}(q^{\ell_j}-q^{-\ell_j})$. Hence the lemma follows.
\end{proof}

By the above lemma we see that $\Uinew\cap \U^{0}=\U^{\imath0}$ can only be satisfied if the parameters satisfy $\va_j=\va_{-j}$ for $j\in I_\circ\backslash \{\pm \pt\}$. From now on we assume the parameters $\{\va_j\}$ always satisfy this condition.

Furthermore, we determine the action of $T_{w_\bu}$ on $\U_\bu$.
\begin{lemma}
\label{lemma:T_wblack}
 For all $j\in I_\bu$, we have
 \begin{equation}
     \begin{aligned}
         T_{w_\bu}(E^{\texttt {X}}_j)=-F_{-j}K_{-j},\quad T_{w_\bu}(F^{\texttt {X}}_j)=-K^{-1}_{-j}E_{-j},\quad
         T_{w_\bu}(K^{\texttt {X}}_j)=K^{-1}_{-j},\\
         T^{-1}_{w_\bu}(E^{\texttt {X}}_j)=-K^{-1}_{-j}F_{-j},\quad T^{-1}_{w_\bu}(F^{\texttt {X}}_j)=-E_{-j}K_{-j},\quad
         T^{-1}_{w_\bu}(K^{\texttt {X}}_j)=K^{-1}_{-j}.
     \end{aligned}
 \end{equation}
\end{lemma}

\begin{proof}
The proof of this lemma follows from the same argument in \cite[Lemma 3.4]{Ko14} and Lemma \ref{lemma:walpha=beta}.
\end{proof}

\subsection{Coideal subalgebra property}
\label{sec:coideal}
One of the key properties of the $\imath$quantum group is that it is a coideal subalgebra of the underlying Hopf algebra rather than a Hopf subalgebra. Here we observe such a structure for $\Uinew$ as well. 

\begin{proposition}
\label{prop:coideal}
$\Uinew$ is a right coideal subalgebra of $\U$. 
\end{proposition}

\begin{proof}
It is not hard to show that $\U_\bu$ and $\U^{\imath0}$ are Hopf subalgebras of $\U$. Thus it suffices to show that 
\begin{equation}
\label{eq:deltanewBj}
\Delta(B_j)\in \Uinew\otimes \U,\ \forall j\in I_\circ.
\end{equation}

Recall $\Delta$ from \eqref{eq:comultinew}. It is straightforward to compute for $j\in I_\circ\backslash\{\pm\pt\}$ that
\begin{equation}
\label{eq:delta1}
    \Delta(B_j)-B_j\otimes K^{-1}_j\in \U^{+}_\bu\U^{\imath0}\otimes \U.
\end{equation}

Now suppose $j=-\pt$. For any reduced expression $\underline w^{(y)}_\bu=s_{y_1}\cdots s_{y_\ell}$ of $w_\bu$, we define \[Y^{(y)}_{t}=s_{y_{t}}\cdots s_{y_\ell}(X),\quad 1\leqslant t\leqslant \ell.\]
More specifically, in this case we choose \[\underline w^{(y)}_\bu=(s_{\pt-1}s_{\pt-2}\cdots s_{-\pt+1})\cdots (s_{\pt-1}s_{\pt-2})(s_{\pt-1}).\]

For convention we drop $(y)$ in the following proof and we define \[\alpha^{ \texttt{D}}_{k}:=\alpha_{\texttt{D},k},\quad \text{ for any }  D\in \mathcal D_{\mathfrak m,\mathfrak n}.\]

Then we observe that
\begin{equation}
    \label{eq:Tw0E-n}
\begin{aligned}
&T_{w_\bu}(E^{\texttt{X}}_{-\pt})=T_{s_{\pt-1}}\cdots T_{s_{-\pt+1}}(E^{\texttt{Y}_{2\pt}}_{-\pt})\\
=& T_{s_{\pt-1}}\cdots T_{s_{-\pt+2}}(E_{-\pt+1}^{\texttt{Y}_{2\pt-1}}E_{-\pt}^{\texttt{Y}_{2\pt-1}}\\
&\quad-(-1)^{p(\alpha_{-\pt+1}^{\texttt{Y}_{2\pt-1}})p(\alpha_{-\pt}^{\texttt{Y}_{2\pt-1}})}q^{(\alpha_{-\pt+1}^{\texttt{Y}_{2\pt-1}},\alpha_{-\pt}^{\texttt{Y}_{2\pt-1}})}E_{-\pt}^{\texttt{Y}_{2\pt-1}}E_{-\pt+1}^{\texttt{Y}_{2\pt-1}})\\
=& T_{s_{\pt-1}}\cdots T_{s_{-\pt+2}}(E_{-\pt+1}^{\texttt{Y}_{2\pt-1}})E_{-n}-zE_{-n}T_{s_{\pt-1}}\cdots T_{s_{-\pt+2}}(E_{-\pt+1}^{\texttt{Y}_{2\pt-1}})
\end{aligned}
\end{equation}
where $z=(-1)^{p(\alpha_{-\pt+1}^{\texttt{Y}_{2\pt-1}})p(\alpha_{-\pt}^{\texttt{Y}_{2\pt-1}})}q^{(\alpha_{-\pt+1}^{\texttt{Y}_{2\pt-1}},\alpha_{-\pt}^{\texttt{Y}_{2\pt-1}})}.$

By \eqref{eq:Tw0E-n} we see that in order to prove \eqref{eq:delta1} for $j=-\pt$, it suffices to prove that
\begin{equation}
    \label{eq:deltagoal}
    \Delta(T_{s_{\pt-1}}\cdots T_{s_{-\pt+1}}(E^{\texttt{Y}_{2\pt}}_{-\pt}))\in T_{s_{\pt-1}}\cdots T_{s_{-\pt+1}}(E^{\texttt{Y}_{2\pt}}_{-\pt})\otimes 1+\U^{+}_\bu\U^{0}K_{\pt}\otimes \U.
\end{equation}
We prove \eqref{eq:deltagoal} by proving the following claim. 

\underline{Claim}: For any $1\leqslant k\leqslant 2\pt-1$, we have
\begin{equation}
    \label{eq:deltagoal'}
    \begin{aligned}
    &\Delta(T_{s_{\pt-1}}\cdots T_{s_{-\pt+k}}(E^{\texttt{Y}_{2\pt+1-k}}_{-\pt+k-1}))\\
    &\qquad \in T_{s_{\pt-1}}\cdots T_{s_{-\pt+k}}(E^{\texttt{Y}_{2\pt+1-k}}_{-\pt+k-1})\otimes 1+\U^{+}_\bu\U^{0}\new^{p(-\pt+k-1)}K_{-\pt+k-1}\otimes \U  .
    \end{aligned}
\end{equation}
We prove \eqref{eq:deltagoal} through induction on $2\pt-1-k$. When $k=2\pt-1$, we have \begin{align*}
\Delta(T_{s_{\pt-1}}(E^{\texttt{Y}_2}_{\pt-2}))=&\Delta(E_{\pt-1}E_{\pt-2}-(-1)^{p(\alpha_{\pt-1})p(\alpha_{\pt-2})}q^{(\alpha_{\pt-1},\alpha_{\pt-2})}E_{\pt-2}E_{\pt-1})\\
\in&T_{s_{\pt-1}}(E^{\texttt{Y}_2}_{\pt-2})\otimes 1+\U^{+}_\bu\U^{0}\new^{p(\pt-2)}K_{\pt-2}\otimes \U .
\end{align*}

Now suppose the claim is true for $k=2$, that is 
\begin{equation}
    \label{eq:deltak=2}
    \Delta(T_{s_{\pt-1}}\cdots T_{s_{-\pt+2}}(E_{-\pt+1}^{\texttt{Y}_{2\pt-1}}))=T_{s_{\pt-1}}\cdots T_{s_{-\pt+2}}(E_{-\pt+1}^{\texttt{Y}_{2\pt-1}})\otimes 1+\sum_\ell x_\ell \new^{p(-\pt+1)}K_{-\pt+1}\otimes y_\ell,
\end{equation}
for some $x_\ell\in \U^{+}_\bu\U^{0},\ y_\ell\in \U$.

In view of \eqref{eq:comultinew}, \eqref{eq:Tw0E-n} and \eqref{eq:deltak=2}, we see that
\begin{equation}
    \label{eq:delta2}
\begin{aligned}
&\Delta(T_{s_{\pt-1}}\cdots T_{s_{-\pt+2}}(E_{-\pt+1}^{\texttt{Y}_{2\pt-1}}))\Delta(E_{-n})\\
\in&T_{s_{\pt-1}}\cdots T_{s_{-\pt+2}}(E_{-\pt+1}^{\texttt{Y}_{2\pt-1}})E_{-n}\otimes 1\\
&\quad+\sum_\ell x_\ell\new^{p(-\pt+1)} K_{-\pt+1}E_{-\pt}\otimes y_\ell+\U^{+}_\bu\U^{0}\new^{p(-\pt)}K_{-n}\otimes \U, \\
&\Delta(E_{-n})\Delta(T_{s_{\pt-1}}\cdots T_{s_{-\pt+2}}(E_{-\pt+1}^{\texttt{Y}_{2\pt-1}}))\\
\in&E_{-n}T_{s_{\pt-1}}\cdots T_{s_{-\pt+2}}(E_{-\pt+1}^{\texttt{Y}_{2\pt-1}})\otimes 1\\
&\quad+\sum_\ell E_{-\pt}x_\ell \new^{p(-\pt+1)}K_{-\pt+1}\otimes y_\ell+\U^{+}_\bu\U^{0}\new^{p(-\pt)}K_{-n}\otimes \U.
\end{aligned}
\end{equation}
By comparing both sides of \eqref{eq:deltak=2} we see that each $x_\ell$ is a monomial of the form $a_{(i)}Z_{i_1}\cdots Z_{i_{2\pt-2}}$ where \[Z_{i_t}\in \{ E_{i_t},\new^{p(i_{t})}K_{i_{t}}\},\quad a_{(i)}\in \Q(q).\] and $\{i_1,\ldots,i_{2\pt-2}\}=\{\pt-1,\ldots,-\pt+2\}$. Hence each $x_\ell\new^{p(-\pt+1)}$ supercommutes with $E_{-\pt}$ as follows:
\[E_{-\pt}x_\ell\new^{p(-\pt+1)}=(-1)^{p(-\pt)}(-1)^{p(\pt-1)+\cdots+p(-\pt+2)}x_\ell E_{-\pt}\new^{p(-\pt+1)}=x_\ell \new^{p(-\pt+1)}E_{-\pt}.\]
The last equality comes from \eqref{eq:assumption}. Moreover, we compute directly that
\begin{equation}
\label{eq:delta3}
     p(\alpha_{-\pt+1}^{\texttt{Y}_{2\pt-1}})=p(s_{-\pt+1}(s_{\pt-1}\cdots s_{-\pt+2})\alpha^{\texttt{X}}_{-\pt+1})=0,\quad (\alpha_{-\pt+1}^{\texttt{Y}_{2\pt-1}},\alpha_{-\pt}^{\texttt{Y}_{2\pt-1}})=(\alpha_{-\pt+1},\alpha_{-\pt}).
\end{equation} 

Now consider $\Delta(T_{w_\bu}(E^{\texttt{X}}_{-\pt}))$, from \eqref{eq:Tw0E-n}, \eqref{eq:delta2} and \eqref{eq:delta3} we have
\begin{equation}
\label{eq:delta4}
\begin{aligned}
    &\Delta(T_{w_\bu}(E^{\texttt{X}}_{-\pt}))=\Delta(T_{s_{\pt-1}}\cdots T_{s_{-\pt+1}}(E^{\texttt{Y}_{2\pt}}_{-\pt}))\\
    =&\Delta(T_{s_{\pt-1}}\cdots T_{s_{-\pt+2}}(E_{-\pt+1}^{\texttt{Y}_{2\pt-1}}))\Delta(E_{-n})-z\Delta(E_{-n})\Delta(T_{s_{\pt-1}}\cdots T_{s_{-\pt+2}}(E_{-\pt+1}^{\texttt{Y}_{2\pt-1}}))\\
    &\quad\in T_{w_\bu}(E^{\texttt{X}}_{-\pt})\otimes 1+\U^{+}_\bu\U^{0}\new^{p(-\pt)}K_{-n}\otimes \U.
\end{aligned}
\end{equation}
Thus the claim is proved. Similarly for $j=\pt$ we also have
\begin{equation}
\label{eq:delta5}
    \Delta(T_{w_\bu}(E^{\texttt{X}}_{\pt}))\in T_{w_\bu}(E^{\texttt{X}}_{\pt})\otimes 1+\U^{+}_\bu\U^{0}K_{\pt}\otimes \U
\end{equation}

Thus we conclude from \eqref{eq:delta1}, \eqref{eq:delta4} and \eqref{eq:delta5} that 
\begin{equation}
\label{eq:comultiBj}
\Delta(B_j)-B_j\otimes K^{-1}_j\in \U^{+}_\bu\U^{\imath0}\otimes \U \quad\forall j\in I_\circ. 
\end{equation}
This proves the proposition. 
\end{proof}

\subsection{Quantum $\imath$Serre relations}
Recall the Serre relations $(R5)-(R11)$ from \eqref{eq:Urelation}. In this subsection we explore the Serre relations of $\Uinew$. For convention, we extend the definition of $B_j$ by setting $B_j=F_j$ for $j\in I_\bu$. 

The triangular decomposition \eqref{eq:triangle} implies an isomorphism between vector spaces
\begin{equation}
    \U^+\otimes \U^0\otimes S(\U^-)\cong \U.
\end{equation}
This leads to a direct sum decomposition 
\begin{equation}
\label{eq:Pprojection}
    \U=\bigoplus_{\mu \in P^\vee}\U^+K_\mu S(\U^-)\oplus\U^+K_\mu\new S(\U^-).
\end{equation}
For any $\mu \in P^\vee$, let $P_\mu:\U \to \U^+K_\mu S(\U^-)\oplus\U^+K_\mu \new S(\U^-)$ denote the projection with respect to \eqref{eq:Pprojection}. Recall \eqref{eq:KQ}, we also use the symbol $P_\lambda$ for $\lambda\in Q$ to denote the projection $P_\lambda:\U \to \U^+K_{\lambda} S(\U^-)\oplus\U^+K_\lambda \new S(\U^-)$ as above.

On the other hand, let $Q^+:=\N \Pi$, we also have the decomposition 
\begin{equation}
\label{eq:piprojection}
    \U=\bigoplus_{\alpha,\beta\in Q^+}U_\alpha^+\U^0\U_{-\beta}^-
\end{equation}
We let $\pi_{\alpha,\beta}:\U\to U_\alpha^+\U^0\U_{-\beta}^-$ denote the projection with respect to \eqref{eq:piprojection}.

The fact that 
\begin{equation}
\label{eq:deltaP}
    \Delta\circ P_\mu(x)=(id\otimes P_\mu)\Delta(x),\ \forall \mu \in P^\vee,x\in \U
\end{equation}
implies the following lemma.
\begin{lemma}{\rm \cite[Lemma 5.9]{Ko14}}
    We have $\Ui=\bigoplus_{\mu \in P^\vee}P_{\mu}(\Ui)$.
\end{lemma}

The following lemma gives the first Serre type relation of $\Ui$.
\begin{lemma}
\label{lem:Bj2}
For $j\in \Iodd$, we have $B_j^2=0$.
\end{lemma}

\begin{proof}
It suffices to check for $j\in I_\circ\cap \Iodd$. Because of \eqref{eq:assumption}, we always have $j\neq -j$. Hence
\begin{align*}
    B_j^2=&(F_j+\va_jT_{w_\bu}(E^{\texttt{X}}_{-j})K_j^{-1})^2\\
    =&F_j^2+\va_j[F_j,T_{w_\bu}(E^{\texttt{X}}_{-j})]K_j^{-1}+\va_j^2 T_{w_\bu}(E^{\texttt{X}}_{-j})^2K_j^{-2}\\
    =0.
\end{align*}
This proves the lemma.
\end{proof}

We define two special weights $\lambda_{j,k}=(1+|(\alpha_j,\alpha_k)|)\alpha_j+\alpha_k$ and $\lambda_j=2\alpha_j+\alpha_{j-1}+\alpha_{j+1}$ in order to apply the projection technique in \cite{Ko14}. Furthermore, we define 
\[S(x_1,x_2)=x_1^2x_2-[2]x_1x_2x_1+x_2x_1^2,\quad \forall x_1,x_2\in \U\]
and recall $S_{t_1,t_2}(x_1,x_2,x_3)$ from \eqref{eq:superSerre}.

\begin{lemma}
\label{lemma:blackSerre}
 (1) Assume $j\in I_\bu,\ k\in I$ and $j\nsim k\in I$, then we have $[B_j,B_k]=B_jB_k-(-1)^{p(j)p(k)}B_kB_j=0$.
 
 (2) For $j\in I_\bu\cap\Ieven$ and $k\sim j$, we have $S(B_j,B_k)=0$.
 
 (3) For $j\in I_\bu\cap\Iodd$ and $k\sim j\sim \ell$, we have $S_{p(k),p(\ell)}(B_k,B_j,B_\ell)=0$.
\end{lemma}

\begin{proof}
(1) In general we have
\begin{align*}
    [B_j,B_k]=&[F_j,F_k]+\va_k[F_j,T_{w_\bu}(E^{\texttt{X}}_{-k})K_k^{-1}].
\end{align*}
Now if $j,k\in I_\bu$, there is nothing to prove. If $k$ is in $I_\circ$ then in this case we can rewrite $F_j$ as $-T_{w_\bu}(E^{\texttt{X}}_{-j})K_{j}^{-1}$ according to Lemma \ref{lemma:T_wblack} and one computes that \[[T_{w_\bu}(E^{\texttt{X}}_{-j})K_{j}^{-1},T_{w_\bu}(E^{\texttt{X}}_{-k})K_k^{-1}]=T_{w_\bu}([E^{\texttt{X}}_{-j}K^{\texttt{X}}_{-j},E^{\texttt{X}}_{-k}K^{\texttt{X}}_{-k}])=0.\]

(2) If $k\in I_\bu$, there is nothing to prove. Thus we can assume that $k\in I_\circ$. In this case we have
\begin{align*}
    F_{j,k}(B_j,B_k)=&F_{j,k}(F_j,\va_kT_{w_\bu}(E^{\texttt{X}}_{-k})K_k^{-1})\\
    =&\va_kF_{j,k}(T_{w_\bu}(E^{\texttt{X}}_{-j})K_{j}^{-1},T_{w_\bu}(E^{\texttt{X}}_{-k})K_k^{-1})\\
    =&\va_k zF_{j,k}(T_{w_\bu}(E^{\texttt{X}}_{-j}),T_{w_\bu}(E^{\texttt{X}}_{-k}))K_j^{-2}K_k^{-1}
\end{align*}
for some $z\in \Z[q,q^{-1}]$. Since $T_{w_\bu}$ is an algebra homomorphism, we see that $F_{j,k}(B_j,B_k)=0$.

(3) Without loss of generality, we assume $k=j-1$ and $\ell=j+1$. Note that only one of $k$ and $\ell$ can belong to $I_\circ$.
Suppose that $k\in I_\bu,\ \ell\in I_\circ$, then we have
\begin{align*}
    &S_{p(k),p(\ell)}(B_k,B_j,B_\ell)\\
    =&\va_\ell S_{p(k),p(\ell)}(F_k,F_j,T_{w_\bu}(E^{\texttt{X}}_{-\ell})K_\ell^{-1})\\
    =&-\va_\ell S_{p(k),p(\ell)}(T_{w_\bu}(E^{\texttt{X}}_{-k})K_k^{-1},T_{w_\bu}(E^{\texttt{X}}_{-j})K_j^{-1},T_{w_\bu}(E^{\texttt{X}}_{-\ell})K_\ell^{-1})\\
    =&-\va_\ell S_{p(k),p(\ell)}(T_{w_\bu}(E^{\texttt{X}}_{-k}),T_{w_\bu}(E^{\texttt{X}}_{-j}),T_{w_\bu}(E^{\texttt{X}}_{-\ell}))K_k^{-1}K_j^{-2}K_\ell^{-1}=0.
\end{align*}
The case $k\in I_\circ,\ \ell\in I_\bu$ can be proved similarly.
\end{proof}

The following technical lemmas provide key steps in the proof of Lemma~\ref{lemma:whiteSerre}.

\begin{lemma}
(1) For any $j\nsim k\in I$, we have $\pi_{0,0}([B_j,B_k])\in \U^{\imath0}$.

(2) For any $j\sim k$ where $j\in \Ieven$, we have $\pi_{0,0}([B_j,B_k])=0$.

    (3) For any $k\sim j\sim \ell$ where $j\in \Iodd$,  we have $\pi_{0,0}(S_{p(k),p(\ell)}(B_k,B_j,B_\ell))=0$.
\end{lemma}

\begin{proof}
(1) The statement follows from Lemma~\ref{lemma:BjB-j}.

(2) The case when $j\in I_\bu$ follows from Lemma~\ref{lemma:blackSerre}. When $j\in I_\circ$, for weight reason we always have $\pi_{0,0}(S_{p(k),p(\ell)}(B_k,B_j,B_\ell))=0$.

(3) Suppose $k<j<\ell$. The case when $j\in I_\bu$ follows from Lemma~\ref{lemma:blackSerre}. When $j\in I_\circ$, at least one of $k$ and $\ell$ lies in $I_\circ$, hence for weight reason we have $\pi_{0,0}(S_{p(k),p(\ell)}(B_k,B_j,B_\ell))=0$.
\end{proof}

\begin{lemma}
\label{lemma:notinXi}
Assume $k\sim j\sim \ell$ and $j\in \Iodd$. Let $\alpha,\beta\in Q^+$, if $\pi_{\alpha,\beta}(S_{p(k),p(\ell)}(B_k,B_j,B_\ell))\neq 0$, then $\lambda_j-\alpha\notin P_\io$ and $\lambda_j-\beta\notin P_\io$.
\end{lemma}

\begin{proof}
    By Lemma~\ref{lemma:blackSerre}, there is nothing to show if $j\in I_\bu$. Hence we may assume that $j\in I_\circ$. Without loss of generality we can suppose that $k<j<\ell$ and $k\in I_\circ$. Consider first the case $\ell \in I_\bu$. We see that \begin{align*}
        &S_{p(k),p(\ell)}(T_{w_\bu}(E^{\texttt{X}}_{-k})K_{k}^{-1},T_{w_\bu}(E^{\texttt{X}}_{-j})K_{j}^{-1},F_\ell)\\
        =&S_{p(k),p(\ell)}(T_{w_\bu}(E^{\texttt{X}}_{-k})K_{k}^{-1},T_{w_\bu}(E^{\texttt{X}}_{-j})K_{j}^{-1},T_{w_\bu}(E^{\texttt{X}}_{-\ell})K_{\ell}^{-1})=0
    \end{align*}
Hence if $\pi_{\alpha,\beta}(S_{p(k),p(\ell)}(B_k,B_j,B_\ell))\neq 0$, we have $0\leq \beta\leq \lambda_j-\alpha_j$ and $0\leq \alpha\leq -\Theta(\lambda_j-\alpha_j)$. This implies that $\lambda_j-\beta\notin P_\io$ and $\lambda_j-\alpha\notin P_\io$. The case when $\ell \in I_\circ$ can be proved similarly as in \cite[Lemma 5.14]{Ko14}.
\end{proof}

For any $J=(j_1,\ldots,j_r)\in I^r$ define $wt(J)=\sum_{i=1}^r\alpha_{j_i}$ and
\begin{equation}
    E_J=E_{j_1}\cdots E_{j_r},\quad F_J=F_{j_1}\cdots F_{j_r},\quad B_J=B_{j_1}\cdots B_{j_r}.
\end{equation}
In this case we also define $|J|=r$.

\begin{lemma}
\label{lemma:whiteSerre}
(1) Assume $j\nsim k\in I$, then we have $P_{-\lambda_{j,k}}([B_j,B_k])=0$.

(2) Assume $j\sim k$ and $j\in \Ieven$, then we have
$P_{-\lambda_{j,k}}(S(B_j,B_k))=0$.

(3) Assume $k\sim j\sim \ell$ and $j\in \Iodd$, then we have
$P_{-\lambda_j}(S_{p(k),p(\ell)}(B_k,B_j,B_\ell))=0$.
\end{lemma}

\begin{proof}
The proofs for all three cases follow the strategy presented in \cite[Proposition 5.16]{Ko14}. Therefore, we will provide the proof for case (3) only since the proofs for cases (1) and (2) can be derived similarly from Lemma~\ref{lemma:blackSerre}, Lemma~\ref{lemma:whiteSerre}, and \cite[Lemma 5.14]{Ko14}.

Assume now $k\sim j\sim \ell$ and $j\in \Iodd$. By Lemma~\ref{lemma:blackSerre} we can assume that $j\in I_\circ$.

Set $\Xi=p(S_{p(k),p(\ell)}(B_k,B_j,B_\ell))$ and $Z=P_{-\lambda_j}(\Xi)$. 
It follows from \eqref{eq:EjBk} and \eqref{eq:comultiBj} that
   \begin{equation}
\label{eq:comultiXi}
    \Delta(\Xi)\in\Xi\otimes K_{-\lambda_j}+\sum_{\{J\mid wt(J)<\lambda_j\}}\U^{+}_\bu\U^{\imath0}B_J\otimes \U.
\end{equation}
Moreover, relations \eqref{eq:deltaP} and \eqref{eq:comultiXi} imply
   \begin{equation}
\label{eq:comultiZ}
    \Delta(Z)\in \Xi\otimes K_{-\lambda_j}+\sum_{\{J\mid wt(J)<\lambda_j\}}\U^{+}_\bu\U^{\imath0}B_J\otimes P_{-\lambda}(\U).
\end{equation}

Assume now that $Z\neq 0$. Choose $\alpha\in Q^+$ maximal such that $\pi_{\alpha,\beta}(Z)\neq 0$ for some $\beta \in Q^+$. In this case by \eqref{eq:comultinew} we have
      \begin{equation}
    \label{eq:P-lambda}
          0\neq (id\otimes \pi_{\alpha,0})\Delta(Z)\in S(U^-) K_{-\lambda+\alpha}\otimes  U_\alpha^+K_{-\lambda}\oplus S(U^-) \new K_{-\lambda+\alpha}\otimes  U_\alpha^+K_{-\lambda}.
      \end{equation}
      Now if $\alpha \neq 0$, the relations \eqref{eq:comultiZ} and \eqref{eq:P-lambda} imply that $K_{-\lambda+\alpha}\in \Ui$, which is in contradiction to Lemma~\ref{lemma:notinXi}.
\end{proof}

\begin{remark}
The assumption \eqref{eq:assumption} is required for Lemma \ref{lemma:blackSerre} and Lemma~ \ref{lemma:whiteSerre}.
\end{remark}

 Let $\mathcal J$ denote a fixed subset of $\cup_{s\in \Z_{\geqslant 0}}I^s$ such that $\{F_J\mid J\in \mathcal J\}$ is a basis of $\U^{-}$. Now we can apply the projection technique to conclude that 
\begin{proposition}
\label{prop:BSerre1}
In $\Uinew$ one has the relation
\begin{equation}
    \begin{aligned}
        (1) &[B_j,B_k]\in \sum_{\{J\mid wt(J)<\lambda_{jk}\}}\U^{+}_\bu\U^{\imath0}B_J, &\text{ for all } j\nsim k\in I, \\
        (2) &S(B_j,B_k)\in \sum_{\{J\mid wt(J)<\lambda_{j,k}\}}\U^{+}_\bu\U^{\imath0}B_J, &\text{ for all } j\sim k\in I,\ j\in \Ieven,\\
        (3) &S_{p(k),p(\ell)}(B_k,B_j,B_\ell)\in\sum_{\{J\mid wt(J)<\lambda_{j}\}}\U^{+}_\bu\U^{\imath0}B_J, &\text{ for all } k\sim j\sim \ell \in I,\ j\in \Iodd.
    \end{aligned}
\end{equation}
\end{proposition}

\begin{proof}
We only proof for (3) and the proof for (1) and (2) is similar.

Since $P_{-\lambda_j}(\Xi)=0$ according to Lemma \ref{lemma:whiteSerre}, by applying the counit to the second tensor factor in \eqref{eq:comultiZ} we get $S_{p(k),p(\ell)}(B_k,B_j,B_\ell)\in\sum_{\{J\mid wt(J)<\lambda_{j}\}}\U^{+}_\bu\U^{\imath0}B_J.$
\end{proof} 

\subsection{Quantum Iwasawa decomposition}
Define a filtration $\mathcal F^*$ of $\U^{-}$ by 
\[\mathcal F^t(\U^{-})=span\{F_J\mid J\in I^s,\ s\leqslant t\},\quad t\in \Z_{\geqslant 0}.\]
As the quantum Serre relations for $\U$ are homogeneous, the set $\{F_J\mid J\in \mathcal J,\ |J|\leqslant t\}$ forms a basis for $\mathcal F^t(\U^{-})$.
\begin{proposition}
\label{prop:basisgeneralpara}
The set $\{B_J\mid J\in \mathcal J\}$ is a basis of the left $\U^{+}\U^{0}$-module $\Uinew$.
\end{proposition}

\begin{proof}
First for any $J\in \mathcal J$ such that $|J|=t$ we have $F_J-B_J\in \U^{+}\U^{0}\mathcal F^{t-1}(\U^{-})$. Thus by induction on $t$ we can conclude that each $F_J$ is contained in the left $\U^{+}\U^{0}$-module generated by $\{B_J\mid J\in \mathcal J\}$.

It remains to show $\{B_J\mid J\in \mathcal J\}$ is linearly independent. Assume there exists a non-empty finite subset $\mathcal J'\subset \mathcal J$ such that $\sum_{J\in \mathcal J'}a_JB_J=0$. Let $t_0=max\{|J|\mid J\in \mathcal J'\}$. In view of the definition of $B_j$, we have
\[\sum_{J\in \mathcal J',|J|=t_0}a_JF_J=0.\]
The linear independence of $\{F_J\mid J\in \mathcal J\}$ implies $a_J=0$ for all $J\in \mathcal J',|J|=t_0$. Then through induction we conclude the desired result. 
\end{proof}

By Proposition \ref{prop:basisgeneralpara} any element in $\Uinew$ can be written as a linear combination of elements in $\{B_J\mid J\in \mathcal J\}$ with coefficients in $\U^{+}\U^{0}$. We want to further show that the coefficients are from $\U^{+}_{\bu}\U^{\imath0}$.
\begin{proposition}
\label{prop:filtration}
The set $\{B_J\mid J\in \mathcal J\}$ is a basis of the left $\U^{+}_{\bu}\U^{\imath0}$-module $\Uinew$.
\end{proposition}

\begin{proof}
First of all, since $\{B_J\mid J\in \mathcal J\}$ is linearly independent over $\U^{+}\U^{0}$, it is also independent over $\U^{+}_{\bu}\U^{\imath0}$.

Secondly, let $L\in I^t$. One can apply the Serre relations (R9) and (R11) in \eqref{eq:Urelation} repeatedly to write
\[F_L=\sum_{J\in\mathcal J, |J|=t}a_JF_J\]
for some $a_J\in \C(q)$. According to Proposition \ref{prop:BSerre1} and Lemma~\ref{lem:Bj2} one sees that
\[B_L-\sum_{J\in\mathcal J, |J|=t}a_JB_J\in\sum_{s<t}\sum_{J\in I^s}\U^{+}_{\bu}\U^{\imath0} B_J.\]
Thus through induction we see that $\{B_J\mid J\in \mathcal J\}$ spans the left $\U^{+}_{\bu}\U^{\imath0}$-module $\Uinew$.
\end{proof}

Define a subspace of $\Ui$ by
\begin{equation}
    \Ui_{\mathcal J}:=\sum_{J\in \mathcal J}\C(q)B_J.
\end{equation}
Then Proposition~\ref{prop:filtration} can be reformulated by saying that the multiplication map
\[\U_\bu^+\otimes \Uio \otimes \Ui_{\mathcal J}\to \Ui\]
is an isomorphism of vector spaces.

Fix a subset $I_\tau\subset I_\circ$ consists of exactly one element of each $\tau$-orbit within $I_\circ$. Let $\Uio_\tau$ denote the subalgebra generated by $\{K_i^{\pm 1}\mid i\in I_\tau\}$. Then we have the following algebra isomorphism
\begin{equation}
\label{eq:UtauUio}
    \Uio_\tau\otimes \Uio \cong \U^0. 
\end{equation}
Define $V_\bu^+$ to be the subalgebra generated by the elements of all
the finite dimensional subspace $\ad(\U_\bu)(E_i)$ for $i\in I_\circ$. It is proved in \cite{K99} that
\begin{equation}
\label{eq:VUb}
    \U^+\cong V_\bu^+\otimes \U_\bu^+.
\end{equation}
The following proposition gives the quantum Iwasawa decomposition of $\U$ associated with $\Ui$.
\begin{theorem}
\label{thm:Iwasawa}
    The multiplication map gives an isomorphism of vector spaces
    \[V_\bu^+\otimes \Uio_\tau \otimes \Ui \cong \U.\]
\end{theorem}

\begin{proof}
    Combining \eqref{eq:VUb} with \eqref{eq:UtauUio} we have an isomorphism
\[ \U^+\U^0\cong V_\bu^+\otimes \Uio_\tau\otimes \Ub^+\otimes \Uio\]
of vector spaces. Thus Proposition~\ref{prop:filtration} implies the desired result.
\end{proof}

\section{$\imath$Schur duality}
\label{sec:duality}
In this section, we explore the fundamental representation $\W$ of $\U$ and establish a commuting action between the $\imath$quantum supergroup $\Uinew$ and the Hecke algebra of type B on $\W^{\otimes d}$.

\subsection{Bimodule structure}
Recall from \eqref{eq:varpi} and \eqref{eq:i_j} that
$$\Pi=w_\bu(\Pi_{\texttt X})=\{ \alpha_i=\epsilon^{\texttt{Y}}_{i-\frac{1}{2}}-\epsilon^{\texttt{Y}}_{i+\frac{1}{2}}\mid i\in I\}$$ is the set of simple roots of $ Y$. 
Recall in \eqref{eq:Iindex} we switch the index set to $I=\I_{\mathfrak m+\mathfrak n-1}$. Another advantage of this notation is that we can naturally parameterize the natural representations of $\U$ by $\I_{\mathfrak m+\mathfrak n}=\I_{2r+m}$.

Let $\W$ denote the natural representation of $\U$. We introduce a new notation for $\I_{\mathfrak m+\mathfrak n}=\I_{2r+\nb}$ to indicate a fixed set partition and to parameterize $\W$ properly:
\begin{equation}
  \label{eq:Ibw}
\Ibw : =\I_{2r+\nb}, \qquad
\Ibw =\Iwl \cup \Ib \cup \Iwr
\end{equation}
where the subsets are
\begin{align}
 \label{eq:III}
\Iwr = \left [\frac{\nb+1}{2}, r+\frac{\nb-1}{2} \right].
\qquad
\Ib = \left [\frac{1-\nb}{2}, \frac{\nb-1}{2} \right], 
\qquad
\Iwl = - \Iwr,
\end{align}
With these notations, the natural representation $\W$ is a vector superspace with an ordered basis
$\{w_a\mid a\in \Ibw\}$ such that 
\begin{align}
\label{eq:natural}
wt(w_a)=\epsilon^{\texttt{Y}}_{a},\quad E_j(w_a)=\delta_{a,j+\frac{1}{2}}w_{a-1},\quad F_j(w_a)=\delta_{a,j-\frac{1}{2}}w_{a+1},\quad \new(w_a)&=(-1)^{p(w_a)}w_a
\end{align}
for all $a\in \Ibw$ and $j\in I$. Note that $p(w_a):=p(\epsilon^{\texttt{Y}}_{a})$.

The Weyl group $W=W_d$ of type $B_d$ is generated by $s_i$, for $0\le i \le d-1$, subject to the Coxeter relations: $s_i^2=1,\  (s_is_{i+1})^3=1,\  (s_0s_1)^4=1$, and $(s_i s_j)^2=1  \ (|i-j|>1)$.

We view $f  \in \Ibw^d$ as a map $f: \{1, \ldots, d\} \rightarrow \Ibw$, and identify 
$f=(f(1), \ldots, f(d))$, with $f(i) \in \Ibw$. 
We define a right action of the Weyl group $W_d$ on $\Ibw^d$ such that,
for $f\in \Ibw^d$ and $0\leq j\leq d-1$,
\begin{equation}
  \label{eq:WB}
f^{s_j} =f\cdot s_j = 
\begin{cases} 
 (\cdots,f(j+1),f(j),\cdots),&\text{ if } j>0; \\
 (-f(1),f(2),\cdots,f(d)), &\text{ if } j=0,\ f(1)\in \Iwl\cup \Iwr; \\
 (f(1),f(2),\cdots,f(d)), &\text{ if } j=0,\ f(1)\in \Ib.
\end{cases} 
\end{equation} 
The only nontrivial relation $(s_0s_1)^4=1$ can be verified by case-by-case inspection depending on whether or not $f(1), f(2) \in \Ib$.
We sometimes write 
$$
f^\omega = f\cdot \omega 
=(f(\omega(1)),\cdots,f(\omega(d))),$$
where it is understood that 
\[
f(\omega(i)) =
 \begin{cases}
 f(\omega(i)), &\text{ if } \omega(i)>0; \\
 f(-\omega(i)), &\text{ if } \omega(i)<0, f(-\omega(i))\in \Ib; \\
 -f(-\omega(i)), &\text{ if } \omega(i)<0,f(-\omega(i))\in \Iwl\cup \Iwr.
\end{cases}
\]

Given $f=(f(1), \ldots, f(d)) \in \Ibw^d$, we denote
\[
M_f = w_{f(1)} \otimes w_{f(2)} \otimes \ldots \otimes w_{f(d)}.
\]
We shall call $f$ a weight and $\{M_f \mid f\in \Ibw^d \}$ the standard basis for $\W^{\otimes d}$.

Since $\Uinew$ is a subalgebra of $\U$, $\W$ is naturally a left $\Uinew$-module. The tensor product $\W^{\otimes d}$ is naturally a $\Ui$-module via the comultiplication $\Delta$. 

The Iwahori-Hecke algebra of type A, denoted by $\Hy_{\mathfrak S_d}$,  is a $\Q(q)$-algebra generated by $H_1,\cdots,H_{d-1}$, subject to the following relations:
\begin{align*}
&(H_i-q)(H_i+q^{-1})=0,\ \ \ \  \\
&H_iH_{i+1}H_i=H_{i+1}H_iH_{i+1},\ \ \ \ \text{for } i\ge 1;\\
&H_iH_j=H_jH_i,\ \ \ \ \text{for }|i-j|>1,
\end{align*}
while the Iwahori-Hecke algebra of type $B$, denoted by $\Hy_{B_d}$, is the $\Q(Q,q)$-algebra generated by $\Hy_{\mathfrak S_d}$ and $H_0$ where
\begin{equation}
    (H_0-Q)(H_0+Q^{-1})=0,\quad H_0H_1H_0H_1=H_1H_0H_1H_0.
\end{equation}
If $\omega \in  W_d$ has a reduced expression $\omega =s_{i_1} \cdots s_{i_k}$, we denote $H_\omega =H_{i_1} \cdots H_{i_k}$. It is well known that $\{H_\omega\mid \omega \in W_d \}$ form a basis for $\Hy_{ B_d}$.

The following proposition is a multi-parameter version of \cite[Proposition 2.10]{CL22}.
\begin{proposition} 
\label{prop:Heckeaction}
There is an right action of $\Hy_{B_d}$ on $\W^{\otimes d}$ as follows:

\[
M_f\cdot H_i=\left\{
\begin{aligned}
&(-1)^{p(w_{f(i)})p(w_{f(i+1)})}M_{f\cdot s_i}+(q-q^{-1})M_{f},\ \ & \text{ if } f(i)<f(i+1),\ i>0;\\
&(-1)^{p(w_{f(i)})p(w_{f(i+1)})}M_{f\cdot s_i},\ \ &\text{ if } f(i)>f(i+1),\ i>0; \\
&\frac{(-1)^{p(w_{f(i)})}(q+q^{-1})+q-q^{-1}}{2}M_f,\ \ & \text{ if } f(i)=f(i+1),\ i>0; \\
&(-1)^{p(w_{f(1)})}M_{f\cdot s_i}+(Q-Q^{-1})M_f, \ \ &\text{ if } f(1)\in \Iwr,\ i=0; \\
&(-1)^{p(w_{f(1)})}M_{f\cdot s_i},\ \ &\text{ if } f(1)\in \Iwl,\ i=0, \\
&QM_f,\ \ & \text{ if } f(1)\in \Ib,\ i=0.
\end{aligned}
\right.
\]

\end{proposition}
\begin{proof}
It has been established in \cite{Mi06} that this defines an action of $\Hy_{\mathfrak S_d}$ on $\W^{\otimes d}$. The remaining nontrivial relation to verify is $H_0H_1H_0H_1=H_1H_0H_1H_0$, which can be confirmed through a case-by-case check.
\end{proof}

\subsection{$\imath$Schur duality of type B}
We first recall results from \cite{Mi06} which establish a type A Schur duality between the quantum supergroup and the Hecke algebra of type A.

 We let ${}^{st}\U_q^\new(\gl(\mathfrak m|\mathfrak n))$ denote the quantum supergroup corresponding to the standard Dynkin diagram as in Example \ref{ex:standard}. The actions we define in Proposition \eqref{prop:Heckeaction} coincides with \cite[(3.1)(3.2)]{Mi06}. We denote by $\Phi^{st}$ (resp. $\Phi$) the homomorphism from ${}^{st}\U_q^\new(\gl(\mathfrak m|\mathfrak n))$ (resp. $\U$) to $End(\W^{\otimes d})$. Both images of $\Phi$ and $\Phi^{st}$ equal to the centralizer of $\HB$-actions within $End(\W^{\otimes d})$, hence we have
$\Phi(\U)=\Phi^{st}({}^{st}\U_q^\new(\gl(m|n)))$. Moreover, we have the following theorem.
\begin{theorem}
\label{thm:typeAduality}
The actions of $\U$ and $\Hy_{\mathfrak S_d}$ on $\W^{\otimes d}$ commute with each other: 
\[
\U \stackrel{\Phi}{\curvearrowright} \W^{\otimes d} \stackrel{\Psi}{\curvearrowleft} \Hy_{\mathfrak S_d}. 
\]
Moreover, $\Phi(\U)$ and $\Psi(\Hy_{\mathfrak S_d})$ form double centralizers in $\End (\W^{\otimes d})$.
\end{theorem}
\begin{proof}
This follows from \cite{Mi06} and the fact that $\Phi(\U)=\Phi^{st}({}^{st}\U_q^\new(\gl(m|n)))$.
\end{proof}

Following the strategy of \cite{SW23}, we develop a type B $\imath$Schur duality between $\Uinew$ and $\HB$. For any reduced expression $\underline w^{(y)}_\bu=s_{y_1}\cdots s_{y_\ell}$ of $w_\bu$, as in the proof of Proposition~\ref{prop:coideal}, we define \[Y^{(y)}_{t}=s_{y_{t}}\cdots s_{y_\ell}(X),\quad 1\leqslant t\leqslant \ell.\]
Note that $Y=Y^{(y)}_{1}$ for any $\underline w^{(y)}_\bu$.

In the next two lemmas we compute explicitly the actions of $ B_j\ (j\in I_\circ)$ on $\W$. 
\begin{lemma}
\label{lem:Tw}
For $a\in \Ibw$ and $i \in I_\circ =[1-n-r, -n] \cup [n, n+r-1]$, we have
\[
T_{w_\bu}(E^{\texttt{X}}_{\tau(i)})(w_a)=\left\{\begin{aligned}
& E_{-i}(w_a), &|i| > \pt;\\
& E_{-\pt+1} E_{-\pt+2}\cdots  E_{\pt-1} E_{\pt}(w_a),&i=-\pt; \\
&\hslash_\nb E_{-\pt} E_{-\pt+1}\cdots  E_{\pt-2} E_{\pt-1}(w_a),& i=\pt.
\end{aligned}\right.
\]
where 
\begin{equation}
\label{eq:hm}
\hslash_\nb=(-1)^{\nb-1+p(\alpha_{-\pt+1}^{Y_{2\pt-1}})p(\alpha_{-\pt}^{Y_{2\pt-1}})+\cdots+p(\alpha_{\pt-1}^{Y_1})p(\alpha_{\pt-2}^{Y_1})}q^{(\alpha^{\texttt{X}}_{-\pt+1},\alpha^{\texttt{X}}_{-\pt})+\cdots+(\alpha^{\texttt{X}}_{\pt-1},\alpha^{\texttt{X}}_{\pt-2})}.
\end{equation}

\end{lemma}

\begin{proof}
The computation follows similarly as in \cite[Lemma 4.2]{SW23}. 

Take $i=n$ for example. We choose
\[\underline w^{(y)}_\bu=(s_{\pt-1}s_{\pt-2}\cdots s_{-\pt+1})\cdots (s_{\pt-1}s_{\pt-2})(s_{\pt-1}).\]

For convention we drop $(y)$ in the following proof. Then we compute $T_{w_\bu}(E^{\texttt{X}}_{- \pt})(w_a)$ as follows:
\begin{align*}
&T_{w_\bu}(E^{\texttt{X}}_{-\pt})(w_a)\\
=&T_{s_{\pt-1}}\cdots T_{s_{-\pt+1}}(E^{Y_{2\pt}}_{-\pt})(w_a)\\
=& T_{s_{\pt-1}}\cdots T_{s_{-\pt+2}}(E_{-\pt+1}^{Y_{2\pt-1}}E_{-\pt}^{Y_{2\pt-1}}\\
&\quad-(-1)^{p(\alpha_{-\pt+1}^{Y_{2\pt-1}})p(\alpha_{-\pt}^{Y_{2\pt-1}})}q^{(\alpha_{-\pt+1}^{Y_{2\pt-1}},\alpha_{-\pt}^{Y_{2\pt-1}})}E_{-\pt}^{Y_{2\pt-1}}E_{-\pt+1}^{Y_{2\pt-1}})w_a \\
=& -(-1)^{p(\alpha_{-\pt+1}^{Y_{2\pt-1}})p(\alpha_{-\pt}^{Y_{2\pt-1}})}q^{(\alpha^{\texttt{X}}_{-\pt+1},\alpha^{\texttt{X}}_{-\pt})}E_{-\pt}T_{s_{\pt-1}}\cdots  T_{s_{-\pt+2}}(E_{-\pt+1}^{Y_{2\pt-1}})(w_a)
\end{align*}
By induction on $n$, we have
$$T_{w_\bu}(E^{\texttt{X}}_{- \pt })(w_a)
= \hslash_\nb E_{-\pt}E_{-\pt+1} \cdots E_{\pt-2}E_{\pt-1}(w_a).
$$
This proves the lemma.
\end{proof}

Lemma~\ref{lem:Tw} together with the formula for $B_j$ immediately imply the following. 
\begin{lemma} 
 \label{lem:Bi}
Let $a\in \Ibw$ and $j \in I_\circ$. The action of $ B_j$ on $\W$ is given by:
$$
 B_{-\pt}(w_a)=
\begin{cases}
w_{-\pt+\frac{1}{2}}, & \text{ if } a=-\pt-\frac{1}{2};\\
\va_{-\pt}w_{-\pt+\frac{1}{2}}, & \text{ if } a=\pt+\frac{1}{2};\\
0, &else,
\end{cases}
$$
$$
 B_i(w_a)=
\begin{cases}
w_{i+\frac{1}{2}}, & \text{ if } a=i-\frac{1}{2};\\
\va_iw_{-i-\frac{1}{2}}, & \text{ if } a=-i+\frac{1}{2};\\
0, &else,\\
\end{cases}
\qquad \text{ for } |i| > \pt,
$$
and (recall $m=2n$)
$$
 B_{\pt}(w_a)=
\begin{cases}
w_{\pt+\frac{1}{2}}+q^{-(-1)^{p\left(w_{\pt-\frac{1}{2}}\right)}}\hslash_\nb\va_{\pt}w_{-\pt-\frac{1}{2}}, & \text{ if } a=\pt-\frac{1}{2};
 \\
0,  &else.
\end{cases}
$$
\end{lemma}

\begin{remark}
When $p(j)=0$ for all $j\in I_\bu$, the computations can be greatly simplified and we have ( $p(w_{a})=0$ for all $a\in \Ibw$ )
\[\hslash_\nb=(-1)^{\nb-1}q^{1-m}.\] 
\end{remark}

For the rest of this section we fix the parameters to be 
 \begin{align}
   \label{eq:para}
\left\{ 
\begin{aligned}
\va_j &=(-1)^{p(j)},\ \text{ if } j\neq \pm\pt, \\
\va_{-\pt} &=(-1)^{p\left(w_{\pt+\frac{1}{2}}\right)}Q,   \qquad\qquad\qquad\qquad\qquad \text{ where } \nb=2\pt \in \Z_{\ge 1}, \\
 \va_{\pt} &=(-1)^{p
\left(w_{\pt+\frac{1}{2}}\right)}q^{(-1)^{p\left(w_{\pt-\frac{1}{2}}\right)}}Q^{-1}\hslash_m^{-1}.
 \end{aligned}
\right.
\end{align}

Introduce the $\Q(Q,q)$-subspaces of $\W$:
\begin{align*}
\W_- &=\bigoplus_{a\in \Iwr}\Q(Q,q)(w_{a} -(-1)^{p(w_{-a})}Qw_{-a}),\qquad
 \W_\bu =\bigoplus_{a\in \I_\bu}\Q(Q,q) w_a,\\
 \W_+ &=\bigoplus_{a\in \Iwr}\Q(Q,q)(w_{a}+(-1)^{p(w_{-a})}Q^{-1}w_{-a}).
\end{align*}

\begin{lemma}
  \label{lem:Vpm}
Assume \eqref{eq:para}. 
Then $\W_-$ and $\W_\bu \oplus \W_+$ are $\Ui$-submodules of $\W$. 
Hence, we have a $\Ui$-module decomposition $\W= (\W_\bu\oplus \W_+) \oplus \W_-$.
\end{lemma}

\begin{proof}
It follows by a direct computation using the formulas \eqref{eq:natural} and Lemma~\ref{lem:Bi}. 
\end{proof}

The decomposition of $\W$ above is also compatible with the $H_0$-action. 

\begin{lemma}
 \label{lem:H0}
The Hecke generator $H_0$ acts on $\W_-$ as $(-Q^{-1} ) \text{Id}$ and acts on $\W_\bu \oplus \W_+$ as $Q\cdot \text{Id}$.
\end{lemma}

\begin{proof}
Follows by Lemma~\ref{lem:Bi}. 
\end{proof}

\begin{theorem}\label{thm:UiHB}
Suppose the parameters satisfy \eqref{eq:para}. 
Then the actions of $\Ui$ and $\HB$ on $\W^{\otimes d}$ commute with each other: 
\[
\Ui \stackrel{\Psi}{\curvearrowright} \W^{\otimes d} \stackrel{\Phi}{\curvearrowleft} \Hy_{B_d}. 
\]
Moreover, $\Psi(\Ui)$ and $\Phi(\HB)$ form double centralizers in $\End (\W^{\otimes d})$.
\end{theorem}

\begin{proof}
By Theorem~\ref{thm:typeAduality}, we know that the actions of $\U$ commute with the action of $H_i$, for $1\le i \le d-1$. Thus, to show the commuting actions of $\Uinew$ and $\HB$, it remains to check the commutativity of the actions of $H_0$ and the generators of $\Uinew$.  

To that end, it suffices to consider $d=1$ (thanks to the coideal property of $\Uinew$ and the fact that the action of $H_0$ depends solely on the first tensor factor). In this case, the commutativity between $\Ui$-action and $H_0$-action on $\W$ follows directly from Lemmas~\ref{lem:Vpm} and \ref{lem:H0}.

The double centralizer property is equivalent to a multiplicity-free decomposition of $\W^{\otimes d}$ as an $\Uinew\otimes \Hy_{B_d}$-module, which reduces by a deformation argument to the $q=1$ setting. At the specialization $q\mapsto 1$, $\Uinew$ becomes the enveloping algebra of a direct sum of two type A Lie superalgebras (cf. \cite{Se83}), $\W =(\W_\bu \oplus \W_+) \oplus \W_-$ becomes the natural representation of it, on which $s_0\in W_d$ acts as $(\text{Id}_{\W_\bu \oplus \W_+}, -\text{Id}_{\W_-})$. The multiplicity-free decomposition of $\W^{\otimes d}$ at $q=1$ can be established by a standard approach as in \cite[Theorem 3.9]{CW12}.
\end{proof}

\section{Quasi $K$-matrix}
\label{sec:Kmatrix}

 From this section on we impose one extra condition on the Satake diagrams we are working with:
\begin{align}
\label{eq:extraassumption}
    p(j)=0,\ \forall j\in I_\bu.
\end{align}

Under the assumption \eqref{eq:extraassumption}, the braid group operators $T_i$ for $i\in I_\bu$ reduce to the ones of Lusztig and we do not need to work with different presentations of $\U$ anymore. Hence the scripts standing for the underlying Dynkin diagrams will be omitted.

In this section, we follow \cite{BK19} and \cite[\S 3.2]{Ko22} to construct the quasi $K$-matrix under the assumption \eqref{eq:extraassumption}.

\subsection{Preparation}

Suppose $Y\in \mathcal D_{\mathfrak m,\mathfrak n}$ is of the form \eqref{eq:AIIIdiagram} and satisfies \eqref{eq:assumption} and \eqref{eq:extraassumption}. Again we let $\U(Y)$ denote the quantum supergroup with generators $\new,E_j,F_j,K_j,j\in I$ associated to $Y$.
Recall that $\Uinew(Y)$ is the $\Q(q)$-subalgebra of $\U(Y)$ generated by $q^\mu \ (\mu\in \cwl),\  E_j, F_j \ (j\in I_\bu)$, $\new$ and
\begin{align}
 B_j=
 F_j+\va_jT_{w_\bu}(E_{\tau j})  K_j^{-1},\ \ \text{for } j\in I_\circ.
\end{align}

Abusing the notation $\tau$, the diagram involution $\tau$ gives rise to the following algebra homomorphism on $\U$:
\begin{proposition}
Under the assumption \eqref{eq:assumption}, there is an involution $\tau$ on $\U$ such that
\begin{equation}
\label{eq:tau}
    \tau(E_j)=E_{\tau j},\quad \tau(F_j)=F_{\tau j},\quad \tau (K_j)=(-1)^{p(j)}K_{\tau j},\quad \tau(\new)=\new
\end{equation}
for all $j\in I$.
\end{proposition}

\begin{proof}
The proof follows from checking on the generators and Lemma~\ref{lemma:ljl-j}. 
\end{proof}

 The super skew derivations (cf. \cite[\S 1.5]{CHW13}) $\lskew$ and on $\U^{+}$ satisfy $\rskew(E_j)=\delta_{i,j},\ \lskew(E_j)=\delta_{i,j},$ and
\begin{equation}
    \label{eq:skew}
    \begin{aligned}
    &\lskew(xy)=(-1)^{p(y)p(i)}\lskew(x)y+q^{(\alpha_i,\mu)}x\lskew(y),\\
    &\rskew(xy)=(-1)^{p(y)p(i)}q^{(\alpha_i,\upsilon)}\rskew(x)y+x\rskew(y)
    \end{aligned}
\end{equation}
for all $x\in \U^{+}_\mu,\ y\in \U^{+}_\upsilon$.

Let $\U^{\geq}$ (resp. $\U^{\leq}$) denote the Hopf subalgebra of $\U$ generated by $\U^0$ and $\U^+$ (resp. $\U^-$). According to \cite[\S 2.4]{Ya94}, there is a non-degenerated bilinear pairing 
$\la \cdot,\cdot \ra$ on $\U^{\leq}\times \U^{\geq}$ such that for all $x,x'\in \U^{\geq}$, $y,y'\in \U^{\leq}$, $\mu,\upsilon\in P$ and $a,b\in \{0,1\}$, we have
\begin{equation}
\label{eq:bilinearform}
\begin{aligned}
    &\la y,xx'\ra=\la \Delta(y),x'\otimes x\ra,\quad &\la yy',x \ra=\la y\otimes y',\Delta(x)\ra, \\
    &\la q^\mu\new^a,q^\upsilon \new^b\ra=(-1)^{ab}q^{-(\mu,\upsilon)},\quad &\la F_j,E_k\ra=\delta_{j,k}, \\
    & \la q^\mu \new^a,E_j\ra =0,\quad &\la F_j,q^\mu \new^a\ra=0
\end{aligned}
\end{equation}

The next lemma is a super analogue of \cite[\S 6.14]{Jan95}.
\begin{lemma}
For all $x\in \U^{+}$, $y\in \U^{-}$ and $j\in I$ one has
\begin{equation}
    \la F_iy,x\ra= (-1)^{p(x)p(j)}\la F_i,E_i\ra\la y,\lskew(x)\ra,\quad \la yF_i,x\ra=\la F_i,E_i\ra \la y,\rskew(x)\ra.
\end{equation}
\end{lemma}

\begin{proof}
Suppose $x\in \U^{+}_\mu$, then we have that
\begin{align*}
    &\Delta(x)=x\otimes 1+\sum_{i\in I} \rskew(x)\new^{p(i)}K_i\otimes E_i+(rest)_1,\\
    &\Delta(x)=\new^{p(x)}K_{\mu}\otimes x+\sum_{i\in I}(-1)^{p(x)p(i)}E_i\new^{p(\mu-\alpha_i)}K_{\mu-\alpha_i}\otimes \lskew(x)+(rest)_2
\end{align*}
where $(rest)_1,(rest)_2\in \new\U^{+}_{\mu-\upsilon}K_\upsilon\otimes \U^{+}_\upsilon$ or $\U^{+}_{\mu-\upsilon}K_\upsilon\otimes \U^{+}_\upsilon$ with $\upsilon>0,\ \upsilon\notin \Pi$. Hence the lemma follows from \eqref{eq:bilinearform}.
\end{proof}

The next lemma is a crucial ingredient to construct the quasi $K$-matrix.
\begin{lemma}
\label{lem:commuF}
For all $x\in \U^{+}$, we have
\begin{equation}
    \label{eq:key}
    [x,F_j]=xF_j-(-1)^{p(x)p(j)}F_jx=\frac{1}{q^{\ell_j}-q^{-\ell_j}}(r_j(x)K_j-K_j^{-1}{}_jr(x)).
\end{equation}
\end{lemma}

\begin{proof}
We induct on $ht(x)$. When $x=E_j$ for some $j\in I$, \eqref{eq:key} follows from the definition. Now if $x=uv$ where $ht(u)<ht(x)$ and $ht(v)<ht(x)$, then we have
\begin{align*}
    uvF_j
    =&(-1)^{p(v)p(i)}uF_j v+\frac{ur_j(v)K_j-uK_j^{-1}{}_jr (v)}{q^{\ell_j}-q^{-\ell_j}} \\
    =&(-1)^{p(uv)p(j)}F_juv+\frac{ur_j(v)K_j-q^{(\alpha_j,|u|)}K_j^{-1}u{}_jr (v)}{q^{\ell_j}-q^{-\ell_j}}\\
    &\qquad+(-1)^{p(v)p(j)}\frac{r_j(u)K_jv-K_j^{-1}{}_jr(u)v}{q^{\ell_j}-q^{-\ell_j}} \\
    =&(-1)^{p(uv)p(j)}F_juv+\frac{r_j(uv)K_j-K_j^{-1}\lskew(uv)}{q^{\ell_j}-q^{-\ell_j}}.
\end{align*}
This proves the lemma.
\end{proof}

\subsection{A recursive formula}
Define $Q^+_{\overline{0}}:=\{\alpha\in Q^+\mid p(\alpha)=0\}$. Extending \eqref{eq:Bi}, we write $B_i=F_i$ for $i\in I_\bu$. Following \cite[\S 6]{BK19}, we establish the following lemma to give equivalent conditions on the existence of the quasi $K$-matrix.

Let $\rho_\bu$ denote the half sum of positive roots of the Levi subalgebra associated with $I_\bu \subset I$.
\begin{lemma}
\label{lem:biup}
Let $\up=\sum_{\mu\in Q_{\overline 0}^+}\up_\mu$ with $\up_\mu\in \U^{+}_\mu$ be an element  in the completion of $\U$, then the following are equivalent.

(1) For all $i\in I$, we have {\rm (cf. \cite[(3.20)]{WZ22})}
\begin{equation}
\label{eq:biup}
    B_i \up =\up \tau \circ \sigma (B_{\tau i}).
\end{equation}

(2) For all $i\in I$, we have
\begin{equation}
\label{eq:upB2}
     B_i \up=\up( F_i+(-1)^{(2\rho_\bu,\alpha_i)}{q}^{(\alpha_i,2\rho_\bu+w_\bu\alpha_{\tau i})}\va_{\tau i}  \overline{T_{w_\bu}( E_{\tau i})} K_i).
\end{equation}

(3) The element $\up$ satisfy the following relations:
\begin{equation}
    \label{eq:derivativeup}
    \begin{aligned}
    &\rskew(\up_\mu)=-(q^{\ell_i}-q^{-\ell_i})\up_{\mu-\alpha_i-w_\bu (\alpha_{\tau i})}(-1)^{(2\rho_\bu,\alpha_i)}{q}^{(\alpha_i,2\rho_\bu+w_\bu\alpha_{\tau i})}\va_{\tau i}  \overline{T_{w_\bu}( E_{\tau i})},\\
    &\lskew(\up_\mu)=-(q^{\ell_i}-q^{-\ell_i})q^{(\alpha_i,w_\bu\alpha_{\tau i})}\va_i T_{w_\bu}( E_{\tau i})\up_{\mu-\alpha_i-w_\bu\alpha_{\tau i}}.
    \end{aligned}
\end{equation}

Moreover, if these relations hold then additionally we have 
\begin{equation}
\label{eq:upblack}
    x\up=\up x \text{ for all }x\in \U^{\imath0}\U_\bu.
\end{equation}
and
\begin{equation}
\label{eq:upnonzero}
    \up_\mu=0 \text{ unless }w_\bu\tau(\mu)=\mu.
\end{equation}
\end{lemma}

\begin{proof}
Note that 
\[\tau \circ \sigma (B_{\tau i})=F_i+\va_{\tau i}K_iT_{w_\bu}^{-1}(E_{\tau i}).\]
Thus the equivalence of (1) and (2) follows from \cite[Lemma 4.17]{BW18b}. The equivalence of (2) and (3) follows from Lemma \ref{lem:commuF}. Moreover, \eqref{eq:upnonzero} follows from an induction argument on $ht(\mu)$; cf. \cite[Proposition 6.1]{BK19}.
\end{proof}

The following lemma states that every non-vanishing term of the quasi $K$-matrix is expected to have parity 0.
\begin{lemma}
For $\mu \in Q^+$, if $w_\bu \tau (\mu)=\mu$, then $p(\mu)=0$.
\end{lemma}

\begin{proof}
Let's use induction on $ht(\mu)$. Now we can write $\mu=\sum_{t=1}^\ell a_t\alpha_{j_t}$ where $a_t>0$ for $1\leqslant t\leqslant \ell$. Now if all $\alpha_{j_t}$ are even roots, then we have $p(\mu)=0$. On the other hand, suppose $p(\alpha_{j_1})=1$. Since $w_\bu \tau (\mu)=\mu$, we have $\mu'=\mu-(\alpha_{j_1}+w_\bu\tau(\alpha_{j_1}))\in Q^+$ and $w_\bu \tau (\mu')=\mu'$ and $ht(\mu')<ht(\mu)$. Thus by the inductive hypothesis we have $p(\mu')=0$. Also, according to \eqref{eq:assumption}, we have $p(\alpha_{j_1})=p(w_\bu\tau (\alpha_{j_1}))$. Thus we have $p(\mu)=0$ as well.
\end{proof}

The system of equations \eqref{eq:derivativeup} for all $i\in I$ provides an equivalent condition for the existence of $\up$, and our objective is to solve it recursively using the following proposition.
\begin{proposition} (cf. \cite[Proposition 6.3]{BK19})
Let $\mu\in Q_{\overline 0}^+$ with $ht(\mu)\geqslant 2$ and fix $A_i,\ {}_iA\in \U^{+}_{\mu-\alpha_i}$ for all $i\in I$. The following are equivalent.

(1) There exists an element $\Xi\in \U^{+}_\mu$ such that $$\rskew(\Xi)=A_i,\quad \lskew(\Xi)={}_iA,\quad \forall i\in I.$$

(2) The elements $A_i$ and ${}_iA$ satisfy the following properties.

\qquad (2a) For all $i,j\in I$, we have \begin{equation}
\label{eq:2a}
r_i({}_jA)=(-1)^{p(i)p(j)}{}_ir(A_j).
\end{equation}

\qquad (2b) For all $i\in \Iodd$, we have
\begin{equation}
\label{eq:2b}
    \la F_i,A_i\ra=0.
\end{equation}

\qquad (2c) For all $i\nsim j\in I$, we have
\begin{equation}
\label{eq:2c}
  \la F_i,A_j\ra=(-1)^{p(i)p(j)}\la F_j,A_i\ra.
\end{equation}

\qquad (2d) For all $i\in \Ieven$ and $j\sim i$, we have 
\begin{equation}
\label{eq:2d}
    \la F_i^2,A_j \ra-[2]\la F_iF_j,A_i \ra+\la F_jF_i,A_i \ra=0.
\end{equation}

\qquad (2e) For all $i\in \Iodd$ and $j\sim i \sim k$, we have 
\begin{equation}
\label{eq:2e}
\begin{aligned}
   [2]\la F_iF_kF_j,A_i\ra=(-1)^{p(j)}\la F_iF_kF_i, A_j\ra+(-1)^{p(j)+p(j)p(k)}\la F_jF_iF_k,A_i\ra\\
   +(-1)^{p(k)}\la F_iF_jF_i, A_k\ra+(-1)^{p(k)+p(j)p(k)}\la F_kF_iF_j,A_i\ra.
   \end{aligned}
\end{equation}
\end{proposition}

\begin{proof}
The proposition follows by a rerun of proof of \cite[Proposition 6.3]{BK19}.
\end{proof}

\subsection{Technical Lemmas}

Define
\begin{equation}
\label{eq:vari'}
    \va_i'=(-1)^{(2\rho_\bu,\alpha_i)}{q}^{(\alpha_i,2\rho_\bu+w_\bu\alpha_{\tau i})}\va_{\tau i}, \text{ for all } i \in I. 
\end{equation}
Thus we can rewrite \eqref{eq:derivativeup} as
\begin{equation}
\label{eq:AiiA}
    \begin{aligned}
       &A_i=-(q^{\ell_i}-q^{-\ell_i})\up_{\mu-\alpha_i-w_\bu \alpha_{\tau i}}\va_i' \overline{T_{w_\bu}(\ E_{\tau i})},\\
       &{}_iA=-(q^{\ell_i}-q^{-\ell_i})q^{(\alpha_i,w_\bu\alpha_{\tau i})}\va_i T_{w_\bu}( E_{\tau i})\up_{\mu-\alpha_i-w_\bu\alpha_{\tau i}}
    \end{aligned}
\end{equation}
In order to construct the quasi $K$-matrix $\up$ recursively, it suffices to show \eqref{eq:AiiA}
 satisfies relations \eqref{eq:2a}--\eqref{eq:2e} for all $i\in I$. Following the strategy from \cite{BK19} we develop several lemmas as follows.

\begin{lemma}
For any $i,j\in I$, we have
\[r_i\circ {}_jr=(-1)^{p(i)p(j)}{}_jr \circ r_i.\]
\end{lemma}

\begin{proof}
If $u=E_k$ or $1$, then we certainly have $r_i\circ {}_jr(u)=(-1)^{p(i)p(j)}{}_jr \circ r_i(u)$. Thus It is enough to show that $r_i\circ {}_jr(xy)={}_jr \circ r_i(xy)$ for any $x\in \U^{+}_\mu,\ y\in \U^{+}_\upsilon$.

We have
\begin{align*}
    &r_i\circ {}_jr(xy)\\
    =&r_i((-1)^{p(y)p(j)}{}_jr(x)y+q^{(\alpha_j,\mu)}x{}_jr(y))\\
    =&(-1)^{p(y)p(j)}[(-1)^{p(y)p(i)}q^{(\alpha_i,\upsilon)}r_i\circ{}_jr(x)y+{}_jr(x)r_i(y)]\\
    &\qquad+q^{(\alpha_j,\mu)}[(-1)^{p({}_jr(y))p(i)}q^{(\alpha_i,\upsilon-\alpha_j)}r_i(x){}_jr(y)+xr_i\circ {}_jr(y)], \\
    &{}_jr\circ r_i(xy)\\
    =&{}_jr((-1)^{p(y)p(i)}q^{(\alpha_i,\upsilon)}\rskew(x)y+x\rskew(y))\\\
    =&(-1)^{p(y)p(i)}q^{(\alpha_i,\upsilon)}[(-1)^{p(y)p(j)}{}_jr\circ r_i(x)y+q^{(\alpha_j,\mu-\alpha_i)}r_i(x){}_jr(y)]\\
    &\qquad+[(-1)^{p({}_ir(y))p(j)}{}_jr(x)r_i(y)+q^{(\alpha_j,\mu)}x{}_jr\circ r_i(y)].
\end{align*}
Now since $p({}_kr(y))=p(y)\pm p(k)$ for any $k\in I$, we have
$
r_i\circ {}_jr=(-1)^{p(i)p(j)}{}_jr \circ r_i.
$
\end{proof}

\begin{lemma}
For all $u\in \U^{+}_\mu$, we have
\begin{equation}
\label{eq:rhori}
\sigma\circ \lskew(u) =(-1)^{p(i)(p(u)+1)}\rskew\circ \sigma(u).
\end{equation}
\end{lemma}

\begin{proof}
We prove by induction on $ht(\mu)$. When $u=E_j$ or $u=1$ the equality holds by definition. Now suppose $x\in \U^{+}_{\mu_1},\ y\in \U^{+}_{\mu_2}$ where $u=xy,\ \mu=\mu_1+\mu_2$ and $\mu_1,\mu_2>0$. Then we have 
\begin{align*}
    \sigma\circ\lskew
    \circ \sigma(xy)
    =&\sigma ((-1)^{p(x)p(i)}\lskew\circ \sigma(y)\sigma(x)+q^{(\alpha_i,\mu_2)}\sigma(y)\lskew\circ \sigma(x)) \\
    =&(-1)^{p(x)p(i)}x\sigma\circ\lskew\circ \sigma(y)+q^{(\alpha_i,\mu_2)}\sigma\circ\lskew\circ \sigma(x)y\\
    =&(-1)^{p(xy)p(i)+p(i)}\rskew(xy).
\end{align*}
This proves the lemma.
\end{proof}

\begin{lemma}
For all $x\in \U^{+}_\mu$, we have
\begin{equation}
\label{eq:bskew}
\bskew(x)=q^{(\alpha_i,\alpha_i-\mu)}\lskew(x).
\end{equation}
\end{lemma}

\begin{proof}
We prove by induction on $ht(\mu)$. When $u=E_j$ or $u=1$ the equality holds by definition. Now suppose $x\in U_{\mu_1}^+,\ y\in \U_{\mu_2}^+$ where $u=xy,\ \mu=\mu_1+\mu_2$ and $\mu_1,\mu_2>0$. Then we have 
\begin{align*}
    \bskew(xy)
    =&x\bskew(y)+(-1)^{p(y)p(i)}q^{-(\alpha_i,\mu_2)}\bskew(x)y \\
    =&q^{(\alpha_i,\alpha_i-\mu_2)}x\lskew(y)+(-1)^{p(y)p(i)}q^{(\alpha_i,\alpha_i-\mu_1-\mu_2)}\lskew(x)y\\
    =&q^{(\alpha_i,\alpha_i-\mu)}[q^{(\alpha_i,\mu_1)}\lskew(y)+(-1)^{p(y)p(i)}\lskew(x)y]\\
    =&q^{(\alpha_i,\alpha_i-\mu)}\lskew(xy).
\end{align*}
This proves the lemma.
\end{proof}

\begin{lemma}
\label{lemma:technical1}
For all $i\in I_\circ$, we have
\begin{equation}
\label{eq:barri}
\overline{\rskew(T_{w_\bu}( E_i))}=(-1)^{(\alpha_i,2\rho_\bu)}q^{(\alpha_i,\alpha_i-w_\bu\alpha_i-2\rho_\bu)}\sigma\circ \tau(r_{\tau i}(T_{w_\bu}( E_{\tau i}))). 
\end{equation}
\end{lemma}

\begin{proof}
Follow from a rerun of the proof of \cite[Lemma 2.9]{BK15}.
\end{proof}

\begin{lemma} 
\label{lemma:technical2}
For all $i\in I_\circ$, we have
\begin{equation}
    \sigma \circ \tau (\rskew(T_{w_\bu}( E_i)))=\rskew(T_{w_\bu}( E_i)).
\end{equation}
\end{lemma}

\begin{proof}
Follow from a rerun of the proof of \cite[Proposition 2.3]{BK15}.
\end{proof}

Combining Lemma \ref{lemma:technical1} and Lemma \ref{lemma:technical2} we get
\begin{corollary}
For all $i\in I_\circ$, we have
\begin{equation}
\label{eq:barricoro}
\overline{\rskew(T_{w_\bu}( E_i))}=(-1)^{(\alpha_i,2\rho_\bu)}q^{(\alpha_i,\alpha_i-w_\bu\alpha_i-2\rho_\bu)}r_{\tau i}(T_{w_\bu}( E_{\tau i})). 
\end{equation}
\end{corollary}

\subsection{Construction of $\up$}
Now we are ready to check that
\eqref{eq:AiiA} for all $i\in I$ indeed satisfy relations \eqref{eq:2a}--\eqref{eq:2e}.

\begin{lemma}
The relation $\rskew({}_jA)=(-1)^{p(i)p(j)}{}_jr(A_i)$ holds for all $i,j\in \I.$
\end{lemma}

\begin{proof}
We calculate that
\begin{align*}
    &\frac{1}{-(q^{\ell_j}-q^{-\ell_j})}\rskew({}_jA)\\
    =&q^{(\alpha_j,w_\bu\alpha_{\tau j})}\va_j[(-1)^{p(i)p(\mu)}q^{(\alpha_i,\mu-\alpha_j-w_\bu\alpha_{\tau j})}\rskew(T_{w_\bu}( E_{\tau j}))\up_{\mu-\alpha_j-w_\bu\alpha_{\tau j}}\\
    &\qquad+T_{w_\bu}( E_{\tau j})\rskew(\up_{\mu-\alpha_j-w_\bu\alpha_{\tau j}})]\\
    =&q^{(\alpha_j,w_\bu\alpha_{\tau j})}\va_j[(-1)^{p(i)p(\mu)}q^{(\alpha_i,\mu-\alpha_j-w_\bu\alpha_{\tau j})}\rskew(T_{w_\bu}( E_{\tau j}))\up_{\mu-\alpha_j-w_\bu\alpha_{\tau j}}\\
    &\qquad-(q^{\ell_i}-q^{-\ell_i})T_{w_\bu}( E_{\tau j})\up_{\mu-\alpha_j-w_\bu\alpha_{\tau j}-\alpha_i-w_\bu\alpha_{\tau i}}\va_i'\overline{T_{w_\bu}( E_{\tau i})}],\\
    &\frac{1}{-(q^{\ell_i}-q^{-\ell_i})}{}_jr(A_i)\\
    =&(-1)^{p(w_\bu\alpha_{\tau i})p(j)}{}_jr(\up_{\mu-\alpha_i-w_\bu\alpha_{\tau i}})\va_i'\overline{T_{w_\bu}( E_{\tau i})}\\
    &\qquad+q^{(\alpha_j,\mu-\alpha_i-w_\bu\alpha_{\tau i})}\va_i'\up_{\mu-\alpha_i-w_\bu\alpha_{\tau i}}{}_jr(\overline{T_{w_\bu}( E_{\tau i})})\\
    =&-(-1)^{p(i)p(j)}(q^{\ell_j}-q^{-\ell_j})q^{(\alpha_j,w_\bu\alpha_{\tau j})}\va_j T_{w_\bu}( E_{\tau j})\cdot\\
    &\up_{\mu-\alpha_j-w_\bu\alpha_{\tau j}-\alpha_i-w_\bu\alpha_{\tau i}}\va_i'\overline{T_{w_\bu}( E_{\tau i})}\\
    &\qquad+q^{(\alpha_j,\mu-\alpha_i-w_\bu \alpha_{\tau i})}\va_i'\up_{\mu-\alpha_i-w_\bu\alpha_{\tau i}}{}_jr(\overline{T_{w_\bu}( E_{\tau i})})
\end{align*}
Recall that $\up_\mu$ vanishes whenever $p(\mu)=1$. By comparing the two equations we see that the relation $\rskew({}_jA)=(-1)^{p(i)p(j)}{}_jr(A_i)$ holds if and only if
\begin{equation}
\label{eq:equivalence1}
\begin{aligned}
&q^{(\alpha_j,w_\bu\alpha_{\tau j})}\va_jq^{(\alpha_i,\mu-\alpha_j-w_\bu\alpha_{\tau j})}\rskew(T_{w_\bu}( E_{\tau j}))\up_{\mu-\alpha_j-w_\bu\alpha_{\tau j}}\\
=&(-1)^{p(i)p(j)}\frac{q^{\ell_j}-q^{-\ell_j}}{q^{\ell_i}-q^{-\ell_i}}q^{(\alpha_j,\mu-\alpha_i-w_\bu\alpha_{\tau i})}\va_i'\up_{\mu-\alpha_i-w_\bu\alpha_{\tau i}}{}_jr(\overline{T_{w_\bu}( E_{\tau i})})
\end{aligned}
\end{equation}
We may assume $i=\tau j$, otherwise both sides of \eqref{eq:equivalence1} vanish. According to \eqref{eq:bskew} we have \[{}_{\tau i}r(\overline{T_{w_\bu}( E_{\tau i})}=q^{(\alpha_{\tau i},w_\bu\alpha_{\tau i}-\alpha_{\tau i})}\overline{r_{\tau i}(T_{w_\bu}( E_{\tau i})}.\]
Substituting this together with Lemma \ref{lemma:ljl-j} we see that \eqref{eq:equivalence1} is equivalent to
\begin{equation}
\label{eq:equivalence2}
    \begin{aligned}
    &\va_{\tau i}q^{(\alpha_i,\mu-\alpha_{\tau i}+w_\bu \alpha_{\tau i}-w_\bu\alpha_i)}\rskew(T_{w_\bu}( E_{i}))\up_{\mu-\alpha_{\tau i}-w_\bu\alpha_{i}}\\
    =&q^{(\alpha_{\tau i},\mu-\alpha_i-\alpha_{\tau i})}\va_i'\up_{\mu-\alpha_i-w_\bu\alpha_{\tau i}}\overline{r_{\tau i}(T_{w_\bu}( E_{\tau i}))}
    \end{aligned}
\end{equation}
Observe that $\mu-\alpha_{\tau i}-w_\bu\alpha_{i}=\mu-\alpha_i-w_\bu\alpha_{\tau i}$ and $w_\bu \alpha_i-w_\bu \alpha_{\tau i}=\alpha_i-\alpha_{\tau i}$. We may further assume that $\up_{\mu-\alpha_{\tau i}-w_\bu\alpha_{ i}}\neq 0$, thus $w_\bu\tau(\mu)=\mu$ and hence $(\alpha_i-\alpha_{\tau i},\mu)=0$. Thus we see \eqref{eq:equivalence2} is equivalent to
\begin{equation}
    \label{eq:equivalence3}
    \begin{aligned}
    &q^{(\alpha_{\tau i},\alpha_i)}\va_{\tau i}\rskew(T_{w_\bu}( E_{i}))=\va_i'\overline{r_{\tau i}(T_{w_\bu}( E_{\tau i}))}\\
    \overset{\eqref{eq:barricoro}}{=}&\va_i'(-1)^{(\alpha_{\tau i},2\rho_\bu)}q^{(\alpha_{\tau i},\alpha_{\tau i}-w_\bu\alpha_{\tau i}-2\rho_\bu)}r_{i}(T_{w_\bu}(E_{i})),
    \end{aligned}
\end{equation}
which follows from the definition of $\va_i'$ \eqref{eq:vari'}.
\end{proof}

The next lemma verifies the relation \eqref{eq:2b}.
\begin{lemma}
For all $i\in I$, we have \[\la F_i,A_i\ra=0.\]
\end{lemma}

\begin{proof}
Since $wt(A_i)=\mu-\alpha_i$. We see that $\la F_i,A_i\ra$ is zero unless $\mu=2\alpha_i$. But in this case we always have $\mu-\alpha_i-w_\bu \alpha_{\tau i}\notin Q^+$. Hence $A_i=0$.
\end{proof}

To verify the relation \eqref{eq:2c}, we have
\begin{lemma}
\label{lemma:verify2c}
For all $i\nsim j\in I$, we have   \[ \la F_i,A_j\ra=(-1)^{p(i)p(j)}\la F_j,A_i\ra.\]
\end{lemma}

\begin{proof}
According to \cite[Lemma 6.4]{BK19}, we can assume that $j=\tau i\in I_\circ\backslash\{\pm\pt\}$ and $\mu=\alpha_i+\alpha_j$, otherwise all terms vanish. In this case we have $\va_i=\va_{j}$,\ 
$A_i=-(q^{\ell_i}-q^{-\ell_i}) \va_i'E_{\tau i}$, and $ A_j=-(q^{\ell_j}-q^{-\ell_j}) \va_j'E_{-j}$. Thus we have
\begin{align*}
    \la F_i,A_{\tau i}\ra=-(q^{\ell_{\tau i}}-q^{-\ell_{\tau i}})\va_{\tau i}'=-(-1)^{p(i)}(q^{\ell_{i}}-q^{-\ell_{i}})\va'_i=(-1)^{p(i)}\la F_{\tau i},A_{i}\ra.
\end{align*}
This proves the lemma.
\end{proof}

To verify the relation \eqref{eq:2d}, we have
\begin{lemma}
 For all $i\in \Ieven$ and $j\sim i$, we have 
\[
    \la F_i^2,A_j \ra-[2]\la F_iF_j,A_i \ra+\la F_jF_i,A_i \ra=0.
\]
\end{lemma}

\begin{proof}
We can assume that $\mu=2\alpha_i+\alpha_j$, otherwise all terms in the above sum vanish. But by \cite[Lemma 6.4]{BK19} in this case we have $w_\bu\circ \tau(\mu)\neq\mu$ for all $j\sim i\in I$. Hence all terms still vanish. 
\end{proof}

To verify the relation \eqref{eq:2e}, we have
\begin{lemma}
 For all $i\in \Iodd$ and $j\sim i \sim k$, we have 
\[
\begin{aligned}
   [2]\la F_iF_kF_j,A_i\ra=(-1)^{p(j)}\la F_iF_kF_i, A_j\ra+(-1)^{p(j)+p(j)p(k)}\la F_jF_iF_k,A_i\ra\\
   +(-1)^{p(k)}\la F_iF_jF_i, A_k\ra+(-1)^{p(k)+p(j)p(k)}\la F_kF_iF_j,A_i\ra.
   \end{aligned}
\]
\end{lemma}

\begin{proof}
Again we may assume that $\mu=2\alpha_i+\alpha_k+\alpha_j$ otherwise all terms vanish. But in this case we see that $w_\bu\circ \tau(\mu)\neq\mu$ unless $\tau j=k,\ \tau i=i$ and $i,j,k\in I_\circ$. However, this is excluded by \eqref{eq:assumption}. Hence all terms still vanish. 
\end{proof}

Therefore, we conclude that \begin{equation*}
    \begin{aligned}
       &A_i=-(q-q^{-1})\up_{\mu-\alpha_i-w_\bu \alpha_{\tau i}}\va_i' \overline{T_{w_\bu}(E_{\tau i})},\\
       &{}_iA=-(q-q^{-1})q^{(\alpha_i,w_\bu\alpha_{\tau i})}\va_i T_{w_\bu}( E_{\tau i})\up_{\mu-\alpha_i-w_\bu\alpha_{\tau i}}
    \end{aligned}
\end{equation*}
for all $i\in I$ satisfy relations \eqref{eq:2a}--\eqref{eq:2e}.

Thus we can conclude the main result of this section.

\begin{theorem}
\label{thm:ibar}
There exists a uniquely determined element $\up=\sum_{\mu\in Q^+_{\overline 0}}\up_\mu$ in the completion of $\U$ with $\up_0=1$ and $\up_\mu\in \U^{+}_\mu$, such that the equality
\[B_i\up=\up( \tau \circ \sigma (B_{\tau i}))\] holds for all $i\in I$.

Moreover, $\up_\mu=0$ unless $w_\bu \tau (\mu) =\mu$.
\end{theorem}

Once $\up$ is constructed, we can define a unique bar involution on $\Ui$ with certain assumption on the parameters as follows. 
\begin{corollary}\label{cor:iba}
Under the assumption that $\overline{\va_j}=\va_j'$ for all $j\in I$, 
there is a unique bar involution $\iba$ on $\Uinew$, defined by $$\iba(x)=\up \overline{x}\up^{-1}, \text{ for all }x\in \Uinew$$ and such that
$$\iba(q)=q{-1},\ \iba(B_j)=B_j,\ \iba(E_k)=E_k,\ \iba(F_k)=F_k,
$$
for $j\in I_\circ, k\in I_{\bu}$. 
\end{corollary}

\begin{proof}
For all $i\in I$, it follows from Lemma \ref{lem:biup} that $B_i\up=\up( \tau \circ \sigma (B_{\tau i}))$ is equivalent to
\eqref{eq:upB2}. Under the assumption $\va_i'=\overline \va_i$ we see that 
\eqref{eq:upB2} is equivalent to
\[B_i\up=\up \overline{B_i}.\]
This concludes the proof.
\end{proof}

\begin{remark}
One can construct $\up$ associated to more general Satake diagrams. For example, one can replace \eqref{eq:extraassumption} by a weaker condition:
\[ p(j)=p(\tau j),\quad \forall j\in I_\bu.\]
Under this assumption formally we still have $w_\bu(Y)=Y$. Thus $T_{w_\bu}$ can still be treated as an automorphism on $\U(Y)$ although it is a composition of both even and odd braid group operators.
\end{remark}

In the last of this subsection we give an example of $\up$.
\begin{example}
Consider the following Satake diagram 
$$
\begin{tikzpicture}[scale=1, semithick]
\node (-2) [label=below:{$-\frac{1}{2}$},scale=0.6] at (4.6,0){$\bigotimes$};
\node (2) [label=below:{$\frac{1}{2}$},scale=0.6] at (5.6,0){$\bigotimes$};
\path  
        (-2) edge (2)
        (-2) edge[dashed,bend left,<->] (2);
\end{tikzpicture}$$

We have $(\alpha_{\frac{1}{2}},\alpha_{-\frac{1}{2}})=1$, $\ell_{-\frac{1}{2}}=1=-\ell_\frac{1}{2}$ and
\begin{align*}
    B_{\frac{1}{2}}=F_{\frac{1}{2}}+\va_{\frac{1}{2}} E_{-\frac{1}{2}}K_{\frac{1}{2}}^{-1},\\
    B_{-\frac{1}{2}}=F_{-\frac{1}{2}}+\va_{-\frac{1}{2}} E_{\frac{1}{2}}K_{-\frac{1}{2}}^{-1}.
\end{align*}
\end{example}

In this case, following the constructions in this section we get
\begin{equation*}
    \up=(\sum_{k\geqslant 0}\frac{(\va_{\frac{1}{2}})^k}{\{k\}!})(E_{\frac{1}{2}}E_{-\frac{1}{2}}+qE_{-\frac{1}{2}}E_{\frac{1}{2}})^k)(\sum_{k\geqslant 0}\frac{(\va_{-\frac{1}{2}})^k}{\{k\}!})(E_{-\frac{1}{2}}E_{\frac{1}{2}}+qE_{\frac{1}{2}}E_{-\frac{1}{2}})^k).
\end{equation*}

\section{$K$-matrix and the $H_0$-action}
\label{sec:UK}
In this section we follow \cite{BW18b} (also cf. \cite{BK19}) to construct a $\Ui$-module intertwiner (or $K$-matrix).

\subsection{$K$-matrix}
Recall the assumption \eqref{eq:extraassumption}. We review several basic lemmas from \cite{BW18b} below. Recall $\sigma$ and $\wp$ from \eqref{eq:rhoandop}. 

\begin{lemma}
For all $i\in I_\bu,j\in I$ and $e=\pm 1$, we have
\begin{equation}
\label{eq:rhoTi}
    \begin{aligned}
        &\wp(T_{i,e}''(E_j))=(-q)^{e(\alpha_i,\alpha_j)}T'_{i,-e}(\wp(E_j)),\\
        &\wp(T_{i,e}'(E_j))=(-q)^{-e(\alpha_i,\alpha_j)}T''_{i,-e}(\wp(E_j)).
    \end{aligned}
\end{equation}
\end{lemma}

\begin{proof}
It follows from a rerun of the proof of \cite[Lemma 4.4]{BW18b}.
\end{proof}

\begin{lemma} 
For $i\in I_\circ$ and $e=\pm 1$, we have
\begin{equation}
\label{eq:rhoTw}
    \begin{aligned}
        &\wp(T_{w_\bu,e}''(E_i))=(-1)^{(2\rho_\bu,\alpha_i)}q^{e(2\rho_\bu,\alpha_i)}T_{w_\bu,-e}'(\wp(E_i)), \\
        &\wp(T_{w_\bu,e}'(E_i))=(-1)^{(2\rho_\bu,\alpha_i)}q^{-e(2\rho_\bu,\alpha_i)}T_{w_\bu,-e}''(\wp(E_i)).
    \end{aligned}
\end{equation}
\end{lemma}

\begin{proof}
    It follows from \eqref{eq:rhoTi} and a rerun of the proof of \cite[Corollary 4.5]{BW18b} .
\end{proof}


Recall $q_i=q^{\ell_i}$.  Following \cite[\S 4.5]{BW18b}, under the assumption \eqref{eq:assumption}, we define the following automorphism of $\U$ obtained by composition $
\vartheta=\sigma \circ \wp \circ \tau
$ such that
\begin{equation}
\label{eq:vartheta}
\begin{aligned}
    \vartheta(E_j)=(-1)^{p(j)}q_{\tau j}F_{\tau j}K^{-1}_{\tau j},\quad &\vartheta(F_j)=(-1)^{p(j)}q^{-1}_{\tau j}K_{\tau j}E_{\tau j},\\
    \vartheta(K_j)=K^{-1}_{\tau j}, \quad&\vartheta(\new)=\new,\quad \text{for all } j\in I.
    \end{aligned}
\end{equation}

For any finite-dimensional $\U$-module $M$, we define a $\U$-module ${}^\vartheta M$ twisted by $\vartheta$ as follows:
\begin{itemize}
    \item ${}^\vartheta M=M$ as an $\Q(q)$-vector space,
    \item We denote a vector in ${}^\vartheta M$ by ${}^\vartheta m$ for $m\in M$,
    \item the action of $u\in \U$ on ${}^\vartheta M$ is given by $\vartheta(u){}^\vartheta m={}^\vartheta (um)$.
\end{itemize}

Let 
\begin{equation}
\label{eq:g}
    g:P\to \Q(q)
\end{equation}
be a function such that for all $\mu \in P$, we have the following two recursive relations of $g$:
\begin{equation}
\label{eq:recursive1}
    g(\mu)=-q_{j}q^{2(\alpha_j, \mu)}g(\mu+\alpha_j),\quad \forall j\in I_\bu.
\end{equation}
\begin{equation}
    \label{eq:recursive3}
    g(\mu)=g(\mu-\alpha_j)(-1)^{p(j)}\va_j(-1)^{(2\rho_\bu,\alpha_j)}q^{(2\rho_\bu,\alpha_j)}q_jq^{(\alpha_{\tau j},w_\bu\mu)}q^{-(\alpha_j,\mu)},\quad \forall j\in I_\circ.
\end{equation}
Such a function $g$ exists; cf. \cite[(4.15)]{BW18b}. Note that under our assumption \eqref{eq:extraassumption}, we have $q_j=q_{\tau j}$ for all $j\in I_\bu$.

\begin{lemma}
For any $\mu \in P$, we have
\begin{equation}
    \label{eq:recursive2}
       g(\mu)=g(\mu-w_\bu \alpha_j)(-1)^{p(j)}\va_{j}q^{(\alpha_{\tau j}, \mu)}q_jq^{-(\alpha_{j},w_\bu\mu)},\quad \forall j\in I_\circ.
\end{equation}
\end{lemma}

\begin{proof}
Recall the following identity \cite[(4.18)]{BW18b}:
\begin{equation}
\label{eq:BW18b4.18}
    g(\mu-\alpha_j)=g(\mu-w_\bu \alpha_j)(-1)^{(2\rho_\bu,\alpha_j)}q^{-(2\rho_\bu,\alpha_j)}q^{2(\alpha_i-w_\bu\alpha_i,\mu)},\quad \forall j\in I_\circ.
\end{equation}
Then applying \eqref{eq:recursive3} to \eqref{eq:BW18b4.18} we get \eqref{eq:recursive2}.
\end{proof}

The function $g$ induces a $\Q(q)$-linear map on any finite dimensional $\U$-module $M$:
\[\tilde g\colon M\to M,\quad \tilde g(m)=g(\mu)m,\quad m\in M_\mu.\]

In the next theorem we construction  the $K$-matrix.
\begin{theorem}{\rm (cf. \cite[Theorem 4.18]{BW18b})}
\label{thm:T}
  For any finite-dimensional $\U$-module $M$, we have the following isomorphism of $\Ui$-modules
  \[\mathcal T:=\up\circ \tilde g\circ T_{w_\bu}^{-1}:M\to {}^\vartheta M.\]
\end{theorem}

\begin{proof}
It suffices to verify that $\mathcal T$ defines a homomorphism of $\Ui$-modules. We shall prove the following identity
\begin{equation}
\label{eq:clubsuit}
    \mathcal T(\vartheta(u)\cdot m)=u\cdot \mathcal T(m),\quad \text{ for } u\in \Ui,\ m\in M_\mu.
\end{equation}

It is straightforward to check \eqref{eq:clubsuit} for $u=K_\mu,\ \new$. Also, for $u=F_j,\ E_j\ (j\in I_\bu)$, the proof are essentially the same as those of \cite[Case(2)-(3),Theorem 4.18]{BW18b}. Thus we only verify for $u=B_j\ (j\in I_\circ)$ as below.

For $u=B_i (i\in I_\circ)$, first we see that 
\begin{align*}
    \vartheta(B_i)=&(-1)^{p(i)}q^{-1}_{\tau i}K_{\tau i}E_{\tau i}+\va_i\sigma \circ \varrho \circ \tau \circ T_{w_\bu}(E_{\tau i})K_{\tau i}\\
    \overset{\eqref{eq:rhoTw}}{=}&(-1)^{p(i)}q^{-1}_{\tau i}K_{\tau i}E_{\tau i}+\va_i(-1)^{(2\rho_\bu,\alpha_i)}{q}^{(2\rho_\bu,\alpha_i)}T_{w_\bu}(\sigma \circ \varrho (E_i))K_{\tau i}\\
    =&(-1)^{p(i)}q^{-1}_{\tau i}K_{\tau i}E_{\tau i}+(-1)^{p(i)}\va_i(-1)^{(2\rho_\bu,\alpha_i)}q^{(2\rho_\bu,\alpha_i)}{q}_iT_{w_\bu}(F_i)T_{w_\bu}(K^{-1}_i)K_{\tau i}.
\end{align*}
Therefore we have
\begin{align*}
    T_{w_\bu}^{-1}\circ \vartheta(B_i)=&(-1)^{p(i)}q^{-1}_{\tau i}T_{w_\bu}^{-1}(K_{\tau i})T_{w_\bu}^{-1}(E_{\tau i})\\
    &\quad+\va_i(-1)^{p(i)}(-1)^{(2\rho_\bu,\alpha_i)}{q}^{(2\rho_\bu,\alpha_i)}q_iF_iK^{-1}_iT_{w_\bu}^{-1}(K_{\tau i}).
\end{align*}
Now we have
\begin{equation*}
\label{eq:lefthand}
    \begin{aligned}
        &\mathcal T(\vartheta(B_i)(m))=\up \circ \tilde g \circ T_{w_\bu}^{-1}(\vartheta (B_i)m)=\up\circ \tilde g\left(T_{w_\bu}^{-1}\circ \vartheta (B_i)(T_{w_\bu}^{-1}(m))\right)\\
        =&\up\circ \tilde g( (-1)^{p(i)}q^{-1}_{\tau i}T_{w_\bu}^{-1}(K_{\tau i})T_{w_\bu}^{-1}(E_{\tau i})T_{w_\bu}^{-1}(m)\\
        &\quad+(-1)^{p(i)}\va_i(-1)^{(2\rho_\bu,\alpha_i)}{q}^{(2\rho_\bu,\alpha_i)}q_iF_iK^{-1}_iT_{w_\bu}^{-1}(K_{\tau i})T_{w_\bu}^{-1}(m) )\\
        =&\up(g(w_\bu \mu+w_\bu \alpha_{\tau i})(-1)^{p(i)}q^{-1}_{\tau i}T_{w_\bu}^{-1}(K_{\tau i})T_{w_\bu}^{-1}(E_{\tau i})T_{w_\bu}^{-1}(m)\\
        &\quad+(-1)^{p(i)}g(w_\bu \mu-\alpha_i)\va_i(-1)^{(2\rho_\bu,\alpha_i)}{q}^{(2\rho_\bu,\alpha_i)}q_iF_iK^{-1}_iT_{w_\bu}^{-1}(K_{\tau i})T_{w_\bu}^{-1}(m) )\\
        =&\up(g(w_\bu \mu+w_\bu \alpha_{\tau i})(-1)^{p(i)}q^{-1}_{\tau i}{q}^{(\alpha_{\tau i},\mu+\alpha_{\tau i})}T_{w_\bu}^{-1}(E_{\tau i})T_{w_\bu}^{-1}(m)\\
        &\quad+(-1)^{p(i)}g(w_\bu \mu-\alpha_i)\va_i(-1)^{(2\rho_\bu,\alpha_i)}{q}^{(2\rho_\bu,\alpha_i)}q_i{q}^{(\alpha_{\tau i},\mu)}{q}^{-(\alpha_i,w_\bu \mu)}F_iT_{w_\bu}^{-1}(m) ).
    \end{aligned}
\end{equation*}

On the other hand, we have
\[B_i\up=\up B_{\tau i}^{\sigma \tau}=\up(F_i+\va_{\tau i}K_iT_{w_\bu}^{-1}(E_{\tau i})).\]
Thus
\begin{equation*}
\label{eq:righthand}
    \begin{aligned}
      B_i \cdot \mathcal T(m)=& B_i\left(\up \circ \tilde g \circ T_{w_\bu}^{-1}(m)\right)\overset{\eqref{eq:biup}}{=}\up (F_i(\tilde g \circ T_{w_\bu}^{-1}(m))+\va_{\tau i}K_iT_{w_\bu}^{-1}(E_{\tau i})(\tilde g \circ T_{w_\bu}^{-1}(m)))\\
    =&\up(g(w_\bu \mu)F_iT_{w_\bu}^{-1}(m)\\
    &\quad+g(w_\bu \mu)\va_{\tau i}{q}^{(\alpha_i,w_\bu \mu+w_\bu\alpha_{\tau i})}T_{w_\bu}^{-1}(E_{\tau i})T_{w_\bu}^{-1}(m)).
    \end{aligned}
\end{equation*}
Now the identity \eqref{eq:clubsuit} for $u=B_j$ follows from comparing the coefficients using \eqref{eq:recursive3} and \eqref{eq:recursive2}.
\end{proof}

\subsection{Realizing $H_0$ via $K$-matrix}
In this subsection the goal is to realize $H_0$-action on $\W$ as in Proposition \ref{prop:Heckeaction} via the  $K$-matrix $\mathcal T$. 

We assume the parameters satisfying \eqref{eq:para} so that the $\imath$Schur duality holds between $\Ui$ and $\HB$, and moreover, $\up$ and $\iba$ in Corollary \ref{cor:iba} uniquely exist. We also assume that $Q\in q^\Z$ thus $\iba(Q)=Q^{-1}$.

Given a $\U$-module $M$, a $\U$-module ${}^\vartheta M$ is simple if and only if $M$ is simple itself. Let $\lambda$ be a dominant integral weight and $L(\lambda)$ be the unique irreducible highest weight module with highest weight vector $\eta_\lambda$. Moreover, we define a lowest weight $\U$-module ${}^\omega L(\lambda)$ of weight $-\lambda$ which has the same underlying vector space as $L(\lambda)$ but with the action twisted by the automorphism $\omega$ where
\begin{equation}
    \omega(E_j)=F_j,\quad \omega(F_j)=(-1)^{p(j)}E_j,\quad \omega(K_\mu)=K_{-\mu}.
\end{equation}
When we consider $\eta_\lambda$ as a vector in ${}^\omega L(\lambda)$, we shall denote it by $\xi_{-\lambda}$.
We check by definition that
\[{}^\vartheta L(\lambda)\cong{}^\omega L(\tau \lambda).\]

A basic example of $L(\lambda)$ is our fundamental representation $\W=L(\epsilon_{-\pt-r+\frac{1}{2}})$. We check by definition that
\begin{equation}
\label{eq:equiv}
    {}^\vartheta \W={}^\vartheta L(\epsilon_{-\pt-r+\frac{1}{2}})\cong {}^\omega L(-\epsilon_{\pt+r-\frac{1}{2}})=L(\epsilon_{-\pt-r+\frac{1}{2}}).
\end{equation}

Recall $\mathcal T=\up\circ \tilde g\circ T_{w_\bu}^{-1}$ and Theorem \ref{thm:T}. Together with \eqref{eq:equiv} we see that $\mathcal T$ induces an $\Ui$-automorphism on $\W$ and send $T_{w_\bu}^{-1}(\eta_{\epsilon_{-\pt-r+\frac{1}{2}}})$ to $\xi_{\epsilon_{\pt+r-\frac{1}{2}}}$, cf. \cite[Theorem 4.18]{BW18b}. Thus we have the following corollary:
\begin{corollary}
\label{cor:KmatrixT}
The $K$-matrix $\mathcal T$ is an $\Ui$-module automorphism of $\W$:
\[\mathcal T:\W\to\W,\qquad w_{-\pt-r+\frac{1}{2}}\mapsto (-1)^{p\left(w_{-\pt-r+\frac{1}{2}}\right)}w_{\pt+r-\frac{1}{2}}.\]
\end{corollary}

\begin{proposition}
\label{prop:realizeH0}
  The action of $H_0$ on $\W^{\otimes d}$ in Proposition \ref{prop:Heckeaction} is realized via the $K$-matrix as $\mathcal T\otimes Id^{\otimes d-1}$.
\end{proposition}

\begin{proof}
According to Corollary \ref{cor:KmatrixT}, we have $\mathcal T(w_{-\pt-r+\frac{1}{2}})=(-1)^{p\left(w_{-\pt-r+\frac{1}{2}}\right)}w_{\pt+r-\frac{1}{2}}$. Recall the parameters satisfy \eqref{eq:para}. 

Suppose $a\in \Iwl$, a simple induction on $a$ shows that
\begin{align*}
    \mathcal T(w_a)=&\mathcal T(B_{a-\frac{1}{2}}B_{a-\frac{3}{2}}\cdots B_{-\pt-r+1} w_{-\pt-r+\frac{1}{2}})\\
    =&B_{a-\frac{1}{2}}B_{a-\frac{3}{2}}\cdots B_{-\pt-r+1}\mathcal T(w_{-\pt-r+\frac{1}{2}})=(-1)^{p(w_a)}w_{-a}=w_a\cdot H_0.
\end{align*}

Now suppose $a=-\pt-\frac{1}{2}$, we have
\begin{align*}
    \mathcal T(w_{-\pt+\frac{1}{2}})=&\mathcal T (B_{-\pt}w_{-\pt-\frac{1}{2}})=B_{-\pt}\mathcal T(w_{-\pt-\frac{1}{2}})=(-1)^{p\left(w_{\pt+\frac{1}{2}}\right)}B_{-\pt}w_{\pt+\frac{1}{2}}\\
    =&Qw_{-\pt+\frac{1}{2}}=w_{-\pt+\frac{1}{2}}\cdot H_0.
\end{align*}
Thus for any $a\in \Ib$, we have
\begin{align*}
     \mathcal T(w_a)=& \mathcal T (F_{a-\frac{1}{2}}F_{a-\frac{3}{2}}\cdots F_{-\pt+1}w_{-\pt+\frac{1}{2}})=F_{a-\frac{1}{2}}F_{a-\frac{3}{2}}\cdots F_{-\pt+1}\mathcal T(w_{-\pt+\frac{1}{2}})\\
     =&Qw_a=w_a\cdot H_0.
\end{align*}

Next we suppose $a=\pt+\frac{1}{2}$, we have
\begin{align*}
    \mathcal T(w_{\pt+\frac{1}{2}})=&\mathcal T(B_\pt(w_{\pt-\frac{1}{2}})-(-1)^{p\left(w_{\pt+\frac{1}{2}}\right)}Q^{-1}w_{-\pt-\frac{1}{2}})\\
    =&QB_n(w_{\pt-\frac{1}{2}})-Q^{-1}w_{\pt+\frac{1}{2}}\\
    =&(-1)^{p\left(w_{\pt+\frac{1}{2}}\right)}w_{-\pt-\frac{1}{2}}+(Q-Q^{-1})w_{\pt+\frac{1}{2}}=w_{\pt+\frac{1}{2}}\cdot H_0.
\end{align*}
Thus for any $a\in \Iwr$, another simple induction on $a$ shows that
\begin{align*}
    \mathcal T(w_a)=&\mathcal T(B_{a-\frac{1}{2}}B_{a-\frac{3}{2}}\cdots B_{\pt+1} w_{\pt+\frac{1}{2}})=B_{a-\frac{1}{2}}B_{a-\frac{3}{2}}\cdots B_{\pt+1}\mathcal T(w_{\pt+\frac{1}{2}})\\
    =&(-1)^{p(w_a)}w_{-a}+(Q-Q^{-1})w_a=w_a\cdot H_0.
\end{align*}
This completes the proof.
\end{proof}

In case $\nb=0$ or $1$, the non-super specialization of Proposition \ref{prop:realizeH0} is established in \cite{BW18a,BWW18}. The property of a $K$-matrix $\mathcal T$ in Corollary \ref{cor:KmatrixT} also provides a conceptual explanation for the commutativity of $H_0$ and $\Ui$ acting on $\W^{\otimes d}$.

\end{document}